\documentclass[a4paper,11pt]{amsart}
\usepackage[plainpages=false]{hyperref}
\usepackage{amsfonts,latexsym,rawfonts,amsmath,amssymb,amsthm, mathrsfs, lscape, calligra}
\usepackage{verbatim}
\usepackage{enumerate}
\usepackage{centernot}
\usepackage{mathtools}
\usepackage{tikz}
\usepackage{tikz-cd}

\usepackage[all]{xy}
\usepackage{graphicx,psfrag}

\usepackage{array, tabularx, color}

\usepackage{setspace}

\newtheorem{thm}{Theorem}[section]
\newtheorem{cor}[thm]{Corollary}
\newtheorem{lem}[thm]{Lemma}

\newtheorem{prop}[thm]{Proposition}

\theoremstyle{remark}
\newtheorem{rmk}[thm]{Remark}

\theoremstyle{definition}
\newtheorem{defi}[thm]{Definition}

\def \R {\mathbb{R}}

\def \C {\mathbb{C}}
\def \E {\mathcal{E}}
\def \F {\mathcal{F}}
\def \O {\mathcal{O}}

\def \H {\mathcal{H}}
\def \db {\bar{\partial}}
\DeclareMathOperator \Sing {Sing}
\DeclareMathOperator\dVol {dVol}
\DeclareMathOperator\Vol {Vol}
\DeclareMathOperator\rank {rank}

\DeclareMathOperator \Id {Id}
\numberwithin{equation}{section}

\newcommand{\RBbb}{\mathbb R}

\newcommand{\Fscr}{\mathscr F}

\newcommand{\Lscr}{\mathscr L}

\newcommand{\SL}{\mathsf{SL}}

\DeclareMathOperator{\tr}{tr}
\DeclareMathOperator{\Tr}{Tr}
\DeclareMathOperator{\Div}{div}

\newcommand{\lra}{\longrightarrow}

\newcommand{\Gr}{{\rm Gr}}

\newcommand{\YM}{{\rm YM}}

\calclayout

\makeatletter
\@namedef{subjclassname@2020}{\textup{2020} Mathematics Subject Classification}
\makeatother

\begin{document}

\title[$\Omega$-Yang-Mills connections]{Compactness for $\Omega$-Yang-Mills connections}

\author[Chen]{Xuemiao Chen}

\author[Wentworth]{Richard A. Wentworth}
\thanks{R.W.'s research is supported in part by NSF grant DMS-1906403.}

\address{Department of Mathematics, University of Maryland, College Park,
MD 20742, USA}

\email{xmchen@umd.edu}
\email{raw@umd.edu}

\subjclass[2020]{Primary: 53C07, 58E15; Secondary: 14D20 }
\keywords{Yang-Mills connections, Gauduchon metric,  Uhlenbeck compactness,  moduli space}

\begin{abstract}
On a Riemannian manifold of dimension $n$ we extend the 
    known analytic  results on Yang-Mills connections to the class of 
    connections called  $\Omega$-Yang-Mills connections, 
    where $\Omega$ is a smooth, not necessarily closed, $(n-4)$-form on $M$. 
    Special cases include $\Omega$-anti-self-dual connections and 
    Hermitian-Yang-Mills connections over general complex manifolds. 
    By a key observation, a weak compactness result is obtained for 
    moduli space of smooth $\Omega$-Yang-Mills connections with 
    uniformly $L^2$ bounded curvature, and it can be improved in the case of 
    Hermitian-Yang-Mills connections over general complex manifolds.
    A removable singularity theorem for singular $\Omega$-Yang-Mills
    connections on a trivial bundle with small energy concentration is also
    proven. 
    As an application,  it is shown how to compactify the moduli space of 
    smooth Hermitian-Yang-Mills connections on unitary bundles over a 
    class of balanced manifolds of Hodge-Riemann type. This class  includes the 
    metrics coming from multipolarizations, and  in particular, the K\"ahler metrics.
    In the case of multipolarizations on a projective algebraic manifold,  
    the compactification of smooth 
    irreducible Hermitian-Yang-Mills connections with fixed determinant
    modulo gauge transformations inherits a complex structure from 
    algebro-geometric considerations.
\end{abstract}

\maketitle

\thispagestyle{empty}
\bibliographystyle{amsplain}

\tableofcontents

\allowdisplaybreaks

\section{Introduction}

\subsection{$\Omega$-Yang-Mills equations}
Let $(M,g)$ be an oriented Riemannian manifold of dimension $n\geq 4$, $\Omega$
a smooth $(n-4)$-form on $M$, and  $E\to M$  a  vector bundle with a
Riemannian metric\footnote{In this paper, 
if $(M,g)$ is a hermitian complex manifold we assume bundles are also complex Hermitian;
otherwise, $E$ can be real or complex.}. The
\emph{$\Omega$-Yang-Mills equations} for a metric connection $A$ on $E$
with curvature $F_A$
are 
\begin{equation} \label{eqn:oym}
    d_A^*\left( F_A+*(F_A\wedge \Omega\right))=0\ ,
\end{equation}
and a solution $A$ to \eqref{eqn:oym} will be called an 
\emph{$\Omega$-Yang-Mills connection} (or $\Omega$-YM connection, for short). 
This equation is the Euler-Lagrange equation of the functional 
\begin{equation} \label{eqn:oym-functional}
    \YM_\Omega(A)=\int_M |F_A|^2\, dV-\int_M \tr(F_A\wedge F_A)\wedge
    \Omega
\end{equation}
which may be viewed as a gauge invariant function on the infinite
dimensional space of metric connections on $E$.  
The first term in \eqref{eqn:oym-functional} 
is the usual Yang-Mills functional $\YM(A)$.
If we assume $\Omega$ is closed, then the second term
in \eqref{eqn:oym-functional} is topological for compact $M$
(or with respect to compactly supported variations), and so the critical
points of $\YM_\Omega$ are identical to those of $\YM$, i.e.\ the Yang-Mills connections.
Indeed, \eqref{eqn:oym} reduces
to $d_A^\ast F_A=0$ in this case. The main goal of this paper is to
extend the analysis of Yang-Mills connections to the more general solutions
of \eqref{eqn:oym} for  the case where $\Omega$ is not closed and
$\Omega$-YM connections are not necessarily Yang-Mills. 

To provide some motivation, let us note 
an interesting special case. We define the 
\emph{$\Omega$-ASD connections} to be the solutions to \eqref{eqn:oym} of the
form
\begin{equation} \label{eqn:asd}
\ast F_A+F_A\wedge \Omega=0
\end{equation}
If $n=4$, $\Omega=1$, then
connections satisfying 
\eqref{eqn:asd} are the much studied  anti-self-dual \emph{instantons}
(cf.\ \cite{FreedUhlenbeck:84, DonaldsonKronheimer:90}).
Higher dimensional  instanton
equations of the type \eqref{eqn:asd} have been considered in a
variety of contexts, and their formulation goes  back to
\cite{CDFN:83}. In the mathematics literature, 
we refer to \cite{DonaldsonThomas:98, Tian:00, DonaldsonSegal:11}, to list 
only  a few of  many recent papers. 
We again point out that an $\Omega$-ASD connection
is not necessarily Yang-Mills unless $\Omega$ is  closed.

If we assume  the comass $|\Omega|\leq 1$, 
then  $\YM_\Omega(A)\geq 0$, and we say $A$ is an absolute
minimizer if $\YM_\Omega(A)=0$.
We have the following simple lemma.

\begin{lem} \label{eqn:minimizer}
Suppose $|\Omega|\leq 1$. Then a connection $A$ is an absolute
    minimizer of $\YM_\Omega$ if and only if it is an $\Omega$-ASD
    connection.
\end{lem}

Now let us suppose that $M$ is an $m$-dimensional hermitian manifold, $2m=n$,
with   K\"ahler form
$\omega$ (not necessarily closed). If the
connection $A$ is integrable (i.e.\ $F_A$ is of type $(1,1)$), then
$$
\YM_\Omega(A)=\int_M|\Lambda F_A|^2\, dV
$$
where $i\Lambda F_A$ is the Hermitian-Einstein tensor, and
$\Omega=\omega^{m-2}/(m-2)!$. It follows that in this case 
the $\Omega$-ASD connections are exactly the  Hermitian-Yang-Mills (HYM)
connections with $i\Lambda F_A=0$. 
In case $\omega$ is a \emph{Gauduchon} metric, then
nontrivial solutions arise from \emph{stable}
holomorphic vector bundles 
 on $M$
(see \cite{LiYau:87})\footnote{HYM connections over hermitian manifolds 
are \emph{not} Yang-Mills connections in general.}. 
Even when  $M$ is a projective algebraic manifold, 
many interesting  examples of solutions can be obtained from  
holomorphic bundles that are stable with respect to 
\emph{multipolarizations} \cite{Miyaoka:87, GrebToma:17}. 
For example, if $\omega_1, \ldots,
\omega_{m-1}$ are K\"ahler forms on $M$, then solutions to the equations
\begin{equation} \label{eqn:multipolarization}
F_A\wedge \omega_1\wedge\cdots\wedge \omega_{m-1}=0
\end{equation}
exist for holomorphic bundles that are stable with respect to $\omega_1, \ldots,
\omega_{m-1}$. On the other hand, 
$\omega_1\wedge\cdots\wedge \omega_{m-1}$ determines a balanced hermitian
metric $\omega$, in general not K\"ahler,  and solutions to 
\eqref{eqn:multipolarization} are $\Omega$-ASD for
$\Omega=\omega^{m-2}/(m-2)!$. 
Note once more that these are not, in general, Yang-Mills, even though the
$\omega_i$ are K\"ahler forms.
Multipolarizations are also considered in more detail in 
\cite{ChenWentworth:21b}. Another motivation is to hopefully give new
nontrivial ways to deform the moduli space of Yang-Mills connections, which
fits into the higher dimensional gauge theoretic picture described in
\cite{DonaldsonSegal:11, DonaldsonThomas:98}. As indicated by the
multipolarization case, the moduli space of HYM
connections can be deformed nontrivially by moving the metric 
on the base complex manifold while at the same time giving a uniform 
$L^2$ bound on the curvature for all the connections. In general, 
we know the K\"ahler condition is often too rigid to deform nontrivially. 
In a sense, the results  obtained here  enrich the picture over 
complex manifolds by providing new structures to consider  as well as
examples arising from algebraic geometry. 

\subsection{Main results}
In this paper, we always assume that $(M, g)$ has bounded geometry in the sense that $(M, g)$ can be isometrically embedded in a larger Riemannian manifold so that $M$ has compact closure. In Section \ref{MonotonicityRegularity}, we will prove a monotonicity
formula and an $\epsilon$-regularity result for $\Omega$-YM connections. 
As a consequence, we obtain  the following version of Uhlenbeck's weak compactness
theorem
 (cf.\ \cite{Nakajima:88,UhlenbeckPreprint}).
\begin{thm}\label{SequentialCompactness}
Let $\{A_i\}$ be a sequence of smooth $\Omega$-YM connections  with
    $\|F_{A_i}\|_{L^2}$ uniformly bounded. Define the set $\Sigma$ by 
$$
\Sigma=\{x\in M: \lim_{r \rightarrow 0^+}\liminf_{i \rightarrow \infty} r^{4-n} \int_{B_r(x)} |F_{A_i}|^2 \geq \epsilon_0^2\}.
$$ 
Then $\Sigma$ is a closed subset of finite $(n-4)$-dimensional Hausdorff
    measure.
There is a bundle $E_\infty\to M\setminus\Sigma$ with a metric 
that is locally isometric to $E$ on $M\setminus \Sigma$. Moreover, there is 
    and  a smooth
    $\Omega$-YM connection $A_\infty$ on $E_\infty$ 
    so that after passing to a subsequence $\{j_i\}$ , and modulo to gauge transformations,
    $A_{j_i}$ converges (locally in the $C^\infty$ topology)
    to an $\Omega$-YM  connection $A_\infty$ outside
    $\Sigma$, i.e. for any compact subset $K\subset M\setminus \Sigma$,
    there exists a sequence of isometries $\Phi_{K}^{j_i}: E_\infty|_{K}
    \rightarrow E|_K$ so that $(\Phi_K^{j_i})^*A_{j_i}$ converges to
    $A_\infty$ smoothly \footnote{Unless otherwise specified, 
    convergence of connections is always taken in this sense.}. Furthermore, at
    each point $x\in \Sigma$, by passing to a subsequence, up to gauge
    transformations,
    $\{\lambda_i^* A_{j_i}\}_i$ converges to a smooth nontrivial $\Omega_x$-YM connection over $\R^n=T_xM$ endowed with the flat metric given by $g_x$. Here $\{\lambda_i\}_i$ denotes a sequence of blow-up rescalings centered at $x$.
\end{thm}

\begin{rmk}
\begin{itemize}
\item As pointed out in \cite{Nakajima:88}, we emphasize here that 
        a priori we only know that
        $E_\infty$ and $E|_{M\setminus \Sigma}$ are isometric 
        on compact subsets away from $\Sigma$. 
        This is due to the possible complexity of the topology of $M\setminus \Sigma$. 
        But as we will see, a global isometry does exist in the case of 
        Hermitian-Yang-Mills connections (see Corollary
        \ref{LimitingBundleOfHE}). This  is due to the fact that we can show $\Sigma$ is a subvariety in this case.

\item A slightly more general statement about the bundle isometries can be obtained as \cite{Waldron:19}. We refer the interested reader there.  
\end{itemize}
\end{rmk}

We will refer to $\Sigma$ as the 
\emph{bubbling set}. By passing to a subsequence, we can assume 
$$
\mu_i:=|F_{A_i}|^2\dVol \rightharpoonup \mu_\infty
$$ 
as a sequence of Radon measures. So the limit of $\{A_i\}_i$ consists of a pair $(A_\infty, \mu_\infty).$ As we will see later (see Lemma \ref{Lem4.2}), $\mu_\infty$ can recover $\Sigma$ intrinsically. We will refer it as $A_i$ sub-converges to $(A_\infty, \mu_\infty)$. 

We also generalize Tian's results \cite{Tian:00} for Yang-Mills connections
 to the case of $\Omega$-YM connections.
\begin{thm} \label{thm:rectifiable}
$\Sigma$ is $(n-4)$-rectifiable.
\end{thm}

Denote $\mathcal A_{\Omega, c}$ to be the space of smooth $\Omega$-YM
connections $A$ on a fixed bundle $E$ with $\|F_{A}\|\leq c$. Now we
consider the space $\overline{\mathcal A_{\Omega, c}}$ by adding limits
$(A_\infty, \mu_\infty)$ of smooth $\Omega$-YM connections $\{A_i\}$ with
$\|F_{A_i}\|_{L^2(M)} \leq c$ (see Section \ref{ModuliCompactness} for more
details.) \emph{Since the space of Radon measures $\{\mu_\infty\}$, which
come from the limits of smooth ones, is compact, we get a natural control
of the singularities of $A_i$.} In particular, the diagonal sequence argument gives the following (see Section \ref{ModuliCompactness} for details)

\begin{thm}\label{Theorem:weak compactness}
$\overline{\mathcal A_{\Omega, c}}$ is weakly sequentially compact in the sense that every sequence $\{(A_i, \mu_i)\}$ in $\overline{\mathcal A_{\Omega, c}}$ sub-converges to some $(A_\infty, \mu_\infty)\in \overline{\mathcal A_{\Omega, c}}$.
\end{thm}

\begin{rmk}
\begin{itemize}
\item Without assuming $A_i$ coming from limits of smooth connections, even in the case of admissible YM connections, we do not know  whether such a limit exists or not due to lack of control of $\text{Sing}(A_i)$. 
\item Again, we emphasize here that  the limiting bundle $E_\infty$ is not known to be isometric to $E|_{M\setminus \Sigma}$ for different subsequences in general. That is why we cannot directly take the quotient of $\mathcal A_{\Omega, c}$ mod gauge here. Due to this, it does not make sense to put a topology on the moduli space at this point. Later in the case of HYM connections over general complex manifolds, the results can be improved. 
\end{itemize}
\end{rmk}

Suppose $A_i$ sub-converges to $(A_\infty,\mu_\infty)$ as above. In Section
\ref{SingularityFormation}, it is straightforward by the argument in
\cite{Tian:00} 
to define a  notion of \emph{bubbling connections} associated to the
sequence. Also the \emph{tangent cones} associated to $(A_\infty,
\mu_\infty)$ are shown to exist. Unlike
\cite{Tian:00} where the tangent cone is defined for stationary admissible
Yang-Mills connections, the tangent cone here is defined for the pair
$(A_\infty, \mu_\infty)$ rather than just for $A_\infty$. This comes from the
fact that a monotonicity formula still holds for the energy density of $\mu_\infty$ which suffices for our use. 

By restricting to the case of $\Omega$-ASD instantons, we can
generalize Tian's results (\cite{Tian:00}) without requiring $\Omega$ be
closed.
\begin{thm}
    $\Omega$ restricts to a volume form of $T_x\Sigma$ at $\mathcal H^{n-4}$ a.e. $x\in \Sigma$.
\end{thm}

In Section \ref{RemovableSingularities},  using the argument in
\cite{SmithUhlenbeck:18},
we generalize the removable singularities theorem for Yang-Mills connections
of Tao-Tian \cite{TaoTian:04} to the case of $\Omega$-YM connections.
\begin{thm}
The removable singularities
    theorem holds for $\Omega$-YM connections on a trivial bundle with small energy concentration away from a closed Hausdorff codimension $4$ set. 
\end{thm}

In the last section, we restrict our discussion to the case of HYM connections over general complex manifolds. If we assume $(A_\infty, \mu)$ is the limit of a sequence of
Hermitian-Yang-Mills connections over a compact Hermitian manifold, then by using
the argument in \cite{Tian:00} for Hermitian-Yang-Mills connections over K\"ahler manifolds
 and the extension theorem in \cite{BandoSiu:94}, 
we can show that
$(A_\infty, \mu)$ are all 
\emph{holomorphic} and $\Sigma$ is a complex subvariety of codimension at least $2$. 
In particular, 
we can now take the quotient of $\overline{\mathcal{A}_{\Omega, c}}$ mod
gauge to get $\overline{M_{HYM,c}}$. There exists a way to give it a
topology that coincides with the four dimensional case (see \cite{DonaldsonKronheimer:90}) so that
\begin{thm}
$\overline{M_{HYM,c}}$ is a first countable sequentially compact Hausdorff space.
\end{thm}

Assume now $(X,\omega)$ is balanced of Hodge-Riemann type (see Section \ref{Section:Hodge-Riemann} for definitions). It turns out there exists a natural $L^2$ bound for the HYM connections in this case. By choosing $c$ large for $\overline{M_{HYM,c}}$, we get the analytic compactification of smooth HYM connections on a fixed unitary bundle, which we denote it as $\overline{M}_{HYM}$. 
\begin{thm}
Over a compact balanced Hermitian manifold of Hodge-Riemann type, $\overline{M}_{HYM}$ is a first countable sequentially compact Hausdorff space.
\end{thm}

\begin{rmk}
Here the Hodge-Riemann type condition on the metrics can give us a
    uniform bound on the curvature of all the $\Omega$-YM connections
    considered. We also refer the interested readers to \cite[Section 3.1
    (Property $B'$)]{DonaldsonSegal:11} where a notion of \emph{taming
    forms} has been introduced for almost $\text{Spin}(7)$ manifold to
    achieve the $L^2$ bound of the curvature as well as a discussion
    reduced to dimension $6$ (see \cite[eqn.\ (28)]{DonaldsonSegal:11}). 
\end{rmk}

By the main results in \cite{Timorin:98}, this gives the following 
\begin{cor}
Over a complex Hermitian manifold $(X, \omega)$ so that $\omega^{m-1}=\omega_0\wedge\cdots \omega_{m-2}$ where $\omega_i$ are positive $(1,1)$ forms with $d\omega^{m-1}=0$ and $d(\omega_1\wedge \cdots \omega_{m-2})=0$, $\overline{M}_{HYM}$ is a first countable sequentially compact Hausdorff space.
\end{cor}

\begin{rmk}
We emphasize here that by \cite{Timorin:98}, $\omega_0\wedge\cdots \omega_{m-2}$ is always  
strictly positive and thus defines a positive $(1,1)$ form on $X$ through $\omega^{m-1}=\omega_0\wedge\cdots \omega_{m-2}$. 
\end{rmk}

In particular, we have
\begin{cor}
Assume $(X,\omega)$ is a compact K\"ahler manifold, $\overline{M_{HYM}}$ is a first countable sequentially compact Hausdorff space.
\end{cor}

\begin{rmk}
\begin{itemize}
\item As mentioned in Theorem \ref{Theorem:weak compactness} above, the
    novelty here is that we do not need to consider a larger space as
        \cite{Tian:00} (explained below). Rather, we  use the crucial condition 
        that the connections considered come from limits of 
        smooth connections. The latter gives a natural control of the 
        singularities of the singular connections on the boundary. 

\item In \cite{Tian:00}, in order to compactify the moduli space, a notion of ideal HYM connection 
    is introduced that generalizes the situation in four dimension 
        (see \cite{DonaldsonKronheimer:90}); namely, those pairs $(A,
        \Sigma)$ with certain natural curvature conditions but not
        necessarily coming from limits of smooth ones. In the case of four
        manifolds, the compactification works essentially due to the good control of the 
        bubbling set,  which consists of points, and Uhlenbeck's removable 
        singularity theorem. In higher dimensions, essential difficulties
        arise  if we insist on such a
        large space of ideal objects.
        One is the lack of  control of $\text{Sing}(A)$. 
        Also, the removable singularity theorem 
does not automatically apply in this situation due to the fact that the
        limiting bundle $E_\infty$, defined only away from the singular
        set, does not necessarily extend to all of $M$ as a
        smooth bundle. 

\item In higher dimensions, and  assuming $(X,\omega)$ is projective,
    it is shown in \cite{GSTW:18} that the space of ideal HYM connections
        modulo gauge is indeed compact. This is essentially due to a boundedness 
        result from the algebraic geometric side which gives control 
        of $\text{Sing}(A)$, and a version of the  removable singularity theorem 
        for HYM connections 
        by Bando and Siu (\cite{BandoSiu:94}). 
        With this, one can take the closure of the space of 
        smooth HYM connections mod gauge in such a space to get a compactification.
 
\item It is an interesting question to find a 
    characterization of the ideal HYM connections added on to  the boundary
        of $\overline{M_{HYM}}$, i.e. determine whether a given ideal HYM connection be approximated by the smooth ones.
\end{itemize}
\end{rmk}
Following from the argument in \cite{GSTW:18}, and using the results on 
compactification of semistable sheaves via multipolarizations in \cite{GrebToma:17}, 
we explain how to give a complex structure to the compactification $\overline{M^*_{HYM}}$, 
where $M^*_{HYM}$
is the moduli space  of smooth irreducible HYM connections with
\emph{fixed determinant}.

Finally, 
 consider   a finite energy HYM connection $A_\infty$ over a complex
Hermitian manifold, and denote by $\E_\infty$ the corresponding
reflexive sheaf. Given the analytic results above 
the following follows directly from the argument in \cite{ChenSun:19}, to which  we refer
the interested reader  for the concepts involved. Here the tangent cone can
be directly defined for $A_\infty$ (not necessarily coming from the limit
of smooth ones).

\begin{thm}
The analytic tangent cone of $A_\infty$ at a point $x$ is uniquely determined by the optimal algebraic tangent cones of $\E_\infty$ at $x$. 
\end{thm}

\medskip
\noindent
{\bf Acknowledgements}. The authors are grateful for comments 
on  this paper from Daniel Greb, Ben Sibley, Song Sun, Matei Toma, and
Thomas Walpuski.  

\section{Sequential compactness of smooth $\Omega$-Yang-Mills connections}\label{MonotonicityRegularity}
\subsection{Monotonicity}
Following the argument used by Price for Yang-Mills connections 
\cite{Price:83}, we will show that a monotonicity formula holds for $\Omega$-YM
connections. We also refer  to
\cite[Thm.\ 2.1.1]{Tian:00} for a slightly more general version 
of the following for Yang-Mills connections. 
\begin{thm}\label{Monotonicity}
There exist positive  constants $a$ and $r_0$,  depending only on the geometry of
    $(M,g)$ and $\Omega$,  with the following significance. If
    $A$ is a smooth solution to \eqref{eqn:oym} and
     $0<r_1<r_2\leq r_0$, then
$$
\begin{aligned}
\int_{B_{r_2}(x) \setminus B_{r_1}(x)} r^{4-n} e^{ar}|\iota_{\partial_r} F_A|^2\leq e^{ar_2} r_2^{4-n}\int_{B_{r_2}(x)} |F_A|^2 -e^{ar_1} r_1^{4-n} \int_{B_{r_1}(x)}  |F_A|^2.
\end{aligned}
$$
\end{thm}

\begin{rmk}
If we denote the scale invariant $L^p$ norms by:
    \begin{equation} \label{eqn:fp}
f_p(x,r):= \left\{r^{2p-n}\int_{B_r(x)}|F_A|^p dV \right\}^{1/p}
    \end{equation}
    then Theorem \ref{Monotonicity}   implies, in particular, that $e^{ar}f_2(x,r)$ 
is increasing for sufficiently small $r$.
\end{rmk}

\begin{proof}[Proof of Theorem \ref{Monotonicity}]
    Let $\pi: P\to M$ be the orthogonal (or unitary) frame bundle of $E$.
Given any connection $B$ on $E$, denote by  $\widetilde{B}$ the associated 
    connection $1$-form
    on the principal bundle $P$. Given a vector field $X$ on $M$ with compact
    support, we denote by
    $\widetilde{X}$  the unique horizontal lift of $X$ to $P$. 
    Let $\widetilde{\Phi}_t$ (resp.\ $\Phi_t$) be the family of
    diffeomorphisms generated by $\widetilde{X}$ (resp. $X$). 
    As  in \cite{Price:83}, we consider the family  of connection
    $1$-forms
    $\widetilde{A}_t=\widetilde{\Phi}_t^*\omega$,  and we  denote by $A_t$
     the corresponding family of connections on $E$. We have 
$$
\delta\widetilde{A}_t(0)=\iota_{\widetilde{X}}d\widetilde{A}=\pi^*\iota_{X} F_A
$$
since $\widetilde{X}$ is the horizontal lift of $X$. In particular, 
$
\delta A_t(0)=\iota_X F_A
$.
Indeed,  choosing a local section $\sigma$ of $P$,  which gives a
    trivialization of $E$, then  by definition: 
$
A_t=\sigma^*\widetilde{A}_t
$.
This implies 
$$
\delta A_t(0)=\sigma^* \iota_{\widetilde{X}} d\widetilde{A}=\sigma^*\pi^* \iota_X F_A=(\pi\sigma)^*\iota_X F_A=\iota_X F_A
$$
since $\pi \sigma=\Id$. Now we look at the variation of the Yang-Mills
    functional along $A_t$. As for this, there are two ways to calculate it.
    First, since $A$ satisfies \eqref{eqn:oym},  we have 
    \begin{equation} \label{eqn:alt-oym}
        d_A^*F_A\pm *(F_A\wedge d\Omega)=0\ .
    \end{equation}
Then,
$$
\begin{aligned}
\frac{d}{dt} \int_{M} |F_{A_t}|^2\bigr|_{t=0}
&=2\int_{M} \langle d_{A} \delta A_t(0), F_A\rangle
=2\int_{M} \langle\iota_X F_A, d_A^* F_A\rangle\\
&=\mp 2\int_{M} \langle\iota_X F_A, *(F_A\wedge d\Omega)\rangle. 
\end{aligned}
$$
    Alternatively, one may differentiate \eqref{eqn:oym-functional} at
    $t=0$ and use the fact that $A$ is critical for $\YM_\Omega$.
In any case, this implies 
    \begin{equation}\label{eqn:first}
\left|\frac{d}{dt} \int_{M} |F_{A_t}|^2\bigr|_{t=0}\right|\leq 2\sup |d\Omega|\int_M|\iota_X F_A| |F_A|.
\end{equation}
Now the second way to calculate the variation is as in \cite{Price:83}. We include the  details here. By definition, we know 
$$
\int_{M} |F_{A_t}|^2
=\int_{M} |F_{A_t}(d\Phi_t \cdot, d\Phi_t \cdot)|^2(\Phi_t \cdot)\, dV
=\int_M  |F_{A_t}(d\Phi_t(e_i), d\Phi_t(e_j))|^2(x) J_{\phi_t^{-1}}\, dV
$$
where $\{e_i\}$ is a  local orthonormal frame near the point $x$.
    Taking derivatives and evaluating at $t=0$ gives 
$$
\begin{aligned}
\frac{d}{dt}\int_{M} |F_{A_t}|^2|_{t=0}
&=\int_M -|F_A|^2 \text{div} X-4\langle F_{A_t}(L_X e_i, e_j), F_{A}(e_i,
    e_j)\rangle\\
&=\int_M -|F_A|^2 \text{div} X+\sum_{i,j}4\int_M \langle 
    F_A(\nabla_{e_i} X, e_j), F_A(e_i, e_j)\rangle \ .
\end{aligned}
$$
Combined with \eqref{eqn:first}, this implies 
\begin{equation}\label{eqn1.2}
\biggl|\int_M -|F_A|^2 \text{div} X+\sum_{i,j}4\int_M \langle 
    F_A(\nabla_{e_i} X, e_j), F_A(e_i, e_j)\rangle\biggr|
    \leq 2\sup |d\Omega|\int_M|\iota_X F_A| |F_A|.
\end{equation}
Near the point $x$ we fix the normal coordinates and let $\{e_1=\partial_r, e_2, \cdots, e_n\}$ be a normal frame. In particular, $\nabla_{\partial_r} \partial_r=0.$ Choose 
$
X=\xi (r)r\partial_r
$,
where $\xi$ is a compact supported function supported over $[0,1+\epsilon]$ with $\xi=1$ on $[0,1]$ and $\xi'\leq 0$. 
Then 
\begin{itemize}
\item $\nabla_{\partial_r} X=(\xi' r+\xi) \frac{\partial}{\partial r}$
\item for $i\geq 2$, $\nabla_{e_i}X=\xi r \nabla_{e_i} \frac{\partial}{\partial r}=\xi e_i+\xi O(r^2)$
\end{itemize}
which implies 
\begin{equation}\label{eqn1.3}
\begin{aligned}
&\sum_{i,j}4\int_M \langle F_A(\nabla_{e_i} X, e_j), F_A(e_i, e_j)\rangle \\
=&\sum_{j}4\int_M \langle
    F_A(\nabla_{\partial_r } X, e_j), F_A(\partial_r , e_j)\rangle+
    \sum_{i\geq 2}\sum_j 4\int_M \langle F_A(\nabla_{e_i} X, e_j), F_A(e_i,
    e_j)\rangle
\\
=&\int_M 4\xi'r |\iota_{\partial_r} F_A|^2+\sum_{j}4\int_M \xi |F_A(\partial_r, e_j)|^2+ \sum_{i\geq 2}\sum_j  4\int_M \xi |F_A(e_i, e_j)|^2+\int_{M} O(r^2)\xi |F_A|^2
\\
=&\int_M 4\xi'r |\iota_{\partial_r} F_A|^2+4\int_M \xi |F_A|^2+\int_{M} O(r^2)\xi |F_A|^2.
\end{aligned}
\end{equation}
and
$$\text{div} X =\xi' r+ n \xi +\xi O(r^2).$$
Given this, we have 
\begin{align}
   \int_M |F_A|^2 \text{div}(X)-2\sup |d\Omega| \int_{M} |X| |F_A|^2 
        &= \int_{M} |F_A|^2 (\xi' r+ n \xi +O(r^2))
        \label{eqn1.4}
\\
        &\qquad\qquad - 2\sup |d\Omega| \int_{M} |X| |F_A|^2 \notag
\end{align}
Plugging eqns.\  \eqref{eqn1.3} and \eqref{eqn1.4} into \eqref{eqn1.2},  we have 
\begin{align}
    \begin{split}\label{eqn1.5}
    \int_{M} |F_A|^2 (\xi' r+ (n-4) \xi +O(r^2))&-2\sup |d\Omega| \int_{M} \xi r |F_A|^2 \\ \leq &\int_M 4\xi'r |\iota_{\partial_r} F_A|^2+\int_{M} O(r^2)\xi |F_A|^2 
    \end{split}
\end{align}
Now by replacing $\xi_{\tau}$ with $\xi_\tau(r)=\xi(\tau^{-1} r)$ in 
\eqref{eqn1.5},
and using the fact that 
$$
\tau \frac{d  \xi_\tau}{d\tau}=-r\xi_\tau'\ ,
$$
we have 
$$
\begin{aligned}
&\int_{M} |F_A|^2(-\tau \frac{d\xi_\tau}{d\tau}+(n-4)\xi_\tau )-2\sup |d\Omega| \int_{M} \xi_\tau r |F_A|^2\\
\leq &-\int_M 4\tau \frac{d\xi_\tau}{d\tau} |\iota_{\partial_r} F_A|^2+\int_{M} O(r^2)\xi_\tau |F_A|^2
\end{aligned}
$$
i.e.  
$$
\begin{aligned}
&\int_{M} |F_A|^2( \tau \frac{d\xi_\tau}{d\tau}+(4-n)\xi_\tau )+2\sup |d\Omega| \int_{M} \xi_\tau r |F_A|^2\\
\geq &\int_M 4\tau \frac{d\xi_\tau}{d\tau} |\iota_{\partial_r} F_A|^2+\int_{M} O(r^2)\xi_\tau |F_A|^2.
\end{aligned}
$$
Multiply the above by $e^{a\tau}\tau^{3-n}$ where $a$ is a constant to be
determined later, 
and use the fact that $\xi_\tau r |F_A|^2\leq \xi_\tau \tau |F_A|^2$,
since $\xi_\tau$ is supported over $\{|x|\leq \tau\}$. We conclude
$$
\begin{aligned}
&e^{a\tau} \frac{d}{d \tau} (\tau^{4-n}\int_{M} \xi_\tau |F_A|^2)+e^{a\tau} \tau^{4-n} 2\sup |d\Omega| \int_{M}\xi_\tau |F_A|^2 \\
\geq & 4 e^{a\tau}\tau^{4-n} \int_{M}\frac{d\xi_\tau}{d\tau} |\iota_{\partial_r} F_A|^2 + e^{a \tau} \tau^{3-n}\int_{M} O(r^2)\xi_\tau |F_A|^2.
\end{aligned}
$$
which implies 
$$
\begin{aligned}
&\frac{d}{d \tau} (e^{a\tau} \tau^{4-n}\int_{M} \xi_\tau |F_A|^2)\\
\geq & 4 e^{a\tau}\tau^{4-n} \int_{M}\frac{d\xi_\tau}{d\tau} |\iota_{\partial_r} F_A|^2 + e^{a \tau} \tau^{3-n}\int_{M} O(r^2)\xi_\tau |F_A|^2+ae^{a\tau} \tau^{4-n}\int_{M} \xi_\tau |F_A|^2\\
&-e^{a\tau} \tau^{4-n} 2\sup |d\Omega| \int_{M}\xi_\tau |F_A|^2
\end{aligned}.
$$
Now choose $a$ large so that $a\gg 2\max\{1, 2\sup |d\Omega| \}$. Since $\frac{d\xi_\tau}{d\tau}=-\frac{r}{\tau}\xi_\tau'$ is nonnegative,
$$
\frac{d}{d \tau} (e^{a\tau} \tau^{4-n}\int_{M} \xi_\tau |F_A|^2)\geq 4 e^{a\tau}\tau^{4-n} \int_{M}\frac{d\xi_\tau}{d\tau} |\iota_{\partial_r} F_A|^2 \geq 4\int_M e^{ar}r^{4-n} \frac{d\xi_\tau}{d\tau} |\iota_{\partial_r} F_A|^2
$$
if $\tau<r_0$ for some $r_0$ so that $e^{a\tau} \tau^{4-n}$ is decreasing
over $[0,r_0]$. By integrating the inequality above from $r_1$ to $r_2$ and
letting $\epsilon \rightarrow 0$, Theorem \ref{Monotonicity} follows. 
\end{proof}

\subsection{$\epsilon$-Regularity}

The goal of this section is to prove the following $\varepsilon$-regularity
result.

\begin{thm} \label{thm:regularity}
There exist positive constants $\epsilon_0$, $r_0$, and $C$, depending 
    only on the geometry of $(M,g)$ and $\Omega$,
    with the following property. 
     If $A$ is a smooth solution to the $\Omega$-Yang-Mills equations
    \eqref{eqn:oym} on $M$, and $x\in M$ is a point for which  
    $f_2(x,r)\leq \epsilon_0$
    for some $0<r\leq r_0$, then
$$
\sup_{B_{r/4}(x)} r^2|F_A|\leq C 
f_2(x,r)
$$
\end{thm}

There are two approaches to the regularity of Yang-Mills equations in
higher dimensions, and both make use of the monotonicity formula. 
Nakajima \cite{Nakajima:88} uses a Bochner-Weitzenb\"ock formula  for the
curvature to directly get the bound in Theorem  \ref{thm:regularity}.
This is similar to Schoen's approach for the harmonic map problem. 
Uhlenbeck \cite{UhlenbeckPreprint}
derives $L^p$ estimates from $L^2$, and then uses a continuity method
to reduce to the case of connections with $L^p$ bounds. 
This has the advantage of applying to a larger class of connections
satisfying curvature bounds rather than equations.
Interestingly, both methods apply directly to the case of $\Omega$-YM
connections, and we find it useful to present each one here.

\subsubsection{Method I}
Suppose $A$ is a smooth solution to \eqref{eqn:oym}.
Then \eqref{eqn:alt-oym} implies 
$$
\Delta_{A} F_A =\mp d_A*(F_A\wedge d\Omega).
$$
In particular, by the  Weitzenb\"ock formula, we have 
\begin{equation}\label{eqn:weitzenbock}
\nabla_A^* \nabla_A F_A = \mp d_A*(F_A\wedge d\Omega)+\{F_A, R_g\}+\{F_A,
    F_A\}\ .
\end{equation}

\begin{prop}
    A solution to \eqref{eqn:oym} satisfies
$$\frac{1}{2}\Delta |F_A|^2\geq -|F_A|^2-c|R_g||F_A|^2-\frac{c^2}{4}
    |d\Omega|^2 |F_A|^2-c |\nabla d\Omega| |F_A|^2$$
    for some constant $c$ depending only on $(M,g)$.
\end{prop}

\begin{proof}
    Indeed, from \eqref{eqn:weitzenbock} we have 
$$
\begin{aligned}
\frac{1}{2}\Delta |F_A|^2
&=-<\nabla_A^* \nabla_A F_A, F_A>+<\nabla_A F_A, \nabla_A F_A>\\
&\geq -|F_A|^3-|R_g||F_A|^2-|d_A*(F_A\wedge d\Omega)||F_A|+|\nabla_A F_A|^2\\
&\geq -|F_A|^3-|R_g||F_A|^2-c(|d\Omega| |\nabla_A F_A||F_A|+|\nabla d\Omega| |F_A|^2)+|\nabla_A F_A|^2\\
&\geq -|F_A|^3-|R_g||F_A|^2-\frac{c^2}{4}|d\Omega|^2 |F_A|^2-c|\nabla d\Omega| |F_A|^2
\end{aligned}
$$
    The last inequality follows from completion of square.
\end{proof}
Given this, we can repeat the argument in  \cite[Lemma 3.1]{Nakajima:88} 
to prove Theorem \ref{thm:regularity}.

\subsubsection{Method II}

Everything is local, so we assume connections are on the trivial bundle in
$\RBbb^n$.
Uhlenbeck's ``good gauge'' theorem states:
\begin{thm}[{\cite[Thm.\ 1.3]{Uhlenbeck:82a}}] \label{thm:good-gauge}
    Fix $n/2<p<n$. 
    There is $\varepsilon_0>0$ and a constant $c_n$
    such that if $A\in L^p_1$ is a connection on $B_1(0)$
    and $f_{n/2}(x,1)<\varepsilon_0$, then $A$ is gauge equivalent to a
    connection (also denoted $A$) satisfying:
    \begin{enumerate}
        \item $d^\ast A=0$;
        \item $\ast A$ vanishes on $\partial B_1(0)$;
        \item $\Vert A\Vert_{L^{n/2}_1}\leq c_n f_{n/2}(0,1)$;
        \item $\Vert A\Vert_{L^p_1}\leq c_n\Vert F_A\Vert_{L^p}$.
    \end{enumerate}
\end{thm}

We will also need
\begin{lem} \label{lem:lp}
    There is $\varepsilon(n)>0$ such that if $A$ is a connection on
    $B_1(0)$ satisfying $\Vert A\Vert_{L^n}\leq\varepsilon(n)$ and items
    (i) and (ii) of the Theorem, then item (iv) holds for all $p$, $n/2\leq
    p<n$. 
\end{lem}

The following result will allow us to go from $L^2$ estimates to $L^p$
estimates.   Let $L^p(x,r):=L^p(B_r(x))$. 

\begin{thm} \label{thm:meyer}
    There are positive constants $\kappa_n, r_0$ and for every 
    for every  $2\leq p<n$, $C_p$,    with the following significance:
    Suppose $A$ is a solution to \eqref{eqn:oym}, and $f_{n/2}(x,r)\leq
    \kappa_n$
     for $r\leq r_0$. 
     Then  
     $$
    f_p(x,r/2)\leq C_p\, f_2(x,r)
     $$
\end{thm}

\begin{proof}
    Rescale to take $r=1$. 
    Use Theorem \ref{thm:good-gauge} and Lemma \ref{lem:lp} for $p=2$ to
    find a gauge where: $d^\ast A=0$, and
    \begin{equation} \label{eqn:21} 
    \Vert A\Vert_{L^2_1(x,1)}\leq C\Vert F_A\Vert_{L^2(x,1)}=C' f_2(x,1)
    \end{equation} 
Now write the equation for the laplacian of $A$ as:
    \begin{align} 
        \Delta A+\{A,dA\}+\{A,A,A\}&=d_A^\ast F_A=\ast(F_A\wedge
        d\Omega)\notag \\
        (\Delta+1) A+\{A,dA\}+\{A,A,A\}&=\ast(dA\wedge d\Omega)
        \label{eqn:A}
    \end{align}
    where the brackets indicate multilinear expressions.
Let $\Lscr$ be  the linear operator acting on $A$ on  the left hand side of
    \eqref{eqn:A}.
    Note that  $L^{n/2}_1\hookrightarrow L^n$, so $[A,A]\in L^{n/2}$, and both $dA$
    and $[A,A]$ are small in $L^{n/2}$. We also have
    $
    L^p_1\times L^{n/2}\hookrightarrow L^p_{-1}
    $.
Hence, 
    we see that $\Lscr=\Lscr_0+\Lscr_1$ is a perturbation of   
    $\Lscr_0:=\Delta+1: L^p_1\to L^p_{-1}$ by $\Lscr_1: L^p_1\to L^p_{-1}$ of
    small norm. 
    As in \cite[p.\ 6]{UhlenbeckPreprint}, a Meyers type
    interior estimate
    for $\Lscr_0$ implies one for $\Lscr$:
    \begin{equation} \label{eqn:interior} 
    \Vert u\Vert_{L^p_1(x,1/2)}\leq C_p(\Vert
    u\Vert_{L^2_1(x,1)}+\Vert \Lscr u\Vert_{L^p_{-1}(x,1)})
    \end{equation}
    where $u=A$.
    Now using \eqref{eqn:21},  the $L^p_{-1}$ norm of the  right hand side 
    of \eqref{eqn:A} is bounded by $f_2(x,1)$ for $p=2n/(n-2)>2$.
    The estimate \eqref{eqn:interior} then gives an improved $L^p_1$ bound on $A$ for $p$ slightly
    bigger than $2$.
    Reiterating this argument, we get $L^p_1$ bounds on $A$ for any $p<n$. 
\end{proof}

Bootstrapping \eqref{eqn:A} gives the estimate:
\begin{equation} \label{eqn:bound}
    \sup_{y\in B_{r/2}(x)}r^2|F_A(y)|\leq C_n\, f_2(x,r)
\end{equation}
Let us fill in some details. First, notice that for $n/2\leq p<n$,
$
L^p_1\times L^p_1\hookrightarrow L^p
$.
Moreover, 
$
L^p_1\times L^p\hookrightarrow L^q
$,
with $q\to n$ as $p\to n$. Hence, from \eqref{eqn:A}
and the $L^p$-elliptic estimate for the Laplacian, we get that $A\in  
L^p_{2,loc}$, for $n/2<p<n$. Again applying multiplication theorems, we get that
$\Delta A\in L^p_1$, and hence, $A\in L^p_{3,loc}$. This implies $A$ is
$C^{1,\alpha}$, and the estimate follows.  

There is one more step:

\begin{lem} \label{lem:better-estimate}
    Suppose $4\rho<r_0$, $f_2(\xi,4\rho)=\varepsilon<\varepsilon_0$.
    Moreover, assume $f_{n/2}(x,r)\leq \kappa_n$ for some
    $r<\rho$. Then:
    \begin{align*}
        f_{n/2}(x,r/2)&\leq C_n\varepsilon \\
        \sup_{y\in B_{r/4}(x)}r^2|F_A(y)|&\leq  K_n\varepsilon
    \end{align*}
\end{lem}

\begin{proof}
    Apply Theorem \ref{thm:meyer}  with $p=n/2$, and use \eqref{eqn:bound}.
\end{proof}

Notice that this Lemma says that once both $f_{n/2}$ and $f_2$ are
sufficiently  small, then
$f_{n/2}$ is even smaller than expected.
Now   Theorem \ref{Monotonicity}  and 
Uhlenbeck's continuity method argument \cite[proof of Thm.\ 1.6]{UhlenbeckPreprint} 
gives the proof of Theorem \ref{thm:regularity}.

\subsection{
Proof of Theorem \ref{SequentialCompactness}}
    This follows from Theorems \ref{Monotonicity} and \ref{thm:regularity} 
    as  in the Yang-Mills case (see
\cite{Nakajima:88, Uhlenbeck:82a}). 

\section{Rectifiability of the blow-up locus}
The results in this section are all local. We will fix a sequence of
$\Omega$-YM connections $A_i$ over $B_{1+\delta_0}:=\{x\in \R^n:
|x|<1+\delta_0\}\subset \R^n$ with $\|F_{A_i}\|_{L^2(B_{1+\delta_0})}$
uniformly bounded and look at the convergence over $B=:B_1$. Here,
$\delta_0>0$ is fixed, and $B_{1+\delta_0}$ is endowed with any fixed
smooth metric with volume form $dV$. 
We assume the standard coordinates are geodesic normal with respect to the
metric. 
Define 
\begin{equation}\label{BS}
\Sigma=\{x\in B: \lim_{r\rightarrow 0^{+}}\liminf_{i} r^{4-n} \int_{B_r(x)}
    |F_A|^2dV \geq \epsilon_0^2\}.
\end{equation}
From the results in the previous section, we only know that $\Sigma$ is a
closed subset of $B$ with locally finite $(n-4)$-Hausdorff measure. We will
show that $\Sigma$ has better structure by generalizing the result in
\cite{Tian:00}; namely, we prove Theorem \ref{thm:rectifiable}. 

The proof closely follows the arguments in \cite{Lin:99, Tian:00}. 
The monotonicity formula obtained in Theorem \ref{Monotonicity} is a key
component. 

\subsection{Elementary properties}
By passing to a subsequence, we can assume
\begin{enumerate}
\item up to gauge transformations, $A_i$ converges to $A_\infty$ locally
    away from  $\Sigma$;
\item $\mu_i:=|F_{A_i}|^2dV$ converges weakly 
    to $\mu$ as a sequence of Radon measures, i.e.\ for any compact supported continuous function $f$, we have 
$$\lim_i \mu_i(f)=\mu(f).$$
\end{enumerate}
By Fatou's lemma, we have 
\begin{equation}
\mu=|F_{A_\infty}|^2 dV +\nu
\end{equation}
for some nonnegative Radon measure $\nu$, which is called the \emph{defect
measure}.

\begin{lem}\label{Lem4.2}
The following properties hold: 
\begin{enumerate}
\item For a.e. $0< r\ll 1$, $\lim_{i} \mu_i(B_r(x))=\mu(B_r(x))$;
\item $r^{4-n}\mu(B_r(x))$ is increasing with $r$. In particular, the function 
$$
\Theta^{n-4}(\mu, x)=\lim_{r\rightarrow 0+} r^{4-n}\mu(B_r(x))
$$
is well-defined, and it is called the energy density of $\mu$ at $x$.
        Furthermore, $\Theta^{n-4}$ is upper semi-continuous and $\mathcal{H}^{n-4}$ approximately continuous at $\mathcal{H}^{n-4}$ a.e. $x\in \Sigma$.
\item $x\in \Sigma$ if and only if $\Theta^{n-4}(\mu, x) \geq \epsilon_0^2$;
\item for $\mathcal{H}^{n-4}$ a.e. $x\in \Sigma$, 
$$
        \limsup_{r\rightarrow 0} r^{4-n}\int_{B_r(x)} |F_{A_\infty}|^2dV=0.
$$
\end{enumerate}
\end{lem}

\begin{proof}
(1) follows from the elementary fact that $\mu(\partial B_r(x))=0$ for a.e.
    $0<r\ll 1$. The first part of (2) now follows from (1) and the fact that
    $r^{4-n}\mu_i(B_r(x))$ increases as $r$ increases. The upper
    semicontinuity follows directly from the monotonicity formula. The
    $\mathcal{H}^{n-4}$ approximate continuity property follows as 
    in \cite[Lemma 3.2.2]{Tian:00} (see also  \cite[p.\ 803]{Lin:99}). For (3), suppose
    $\Theta^{n-4}(\mu, x) \geq \epsilon_0^2$, obviously, $x\notin \Sigma$.
    Now suppose $x\in \Sigma$, if $\Theta^{n-4}(\mu, x) < \epsilon_0^2$, by
    $(1)$, $\mu_i(B_r(x))<\epsilon_0^2$ for $0<r\ll 1$. By
    $\epsilon$-regularity, $A_i$ converges smoothly near $x$ which implies
    $x\notin \Sigma$. This is a contradiction. For (4), see     \cite[p.\ 222]{Tian:00}.
\end{proof}

\begin{rmk}
From this, we know $\Sigma=\{x\in B: \Theta^{n-4}(\mu, x) \geq
    \epsilon_0^2\}$, which recovers the statement that $\Sigma$ a closed
    subset of $B$ of finite $(n-4)$-dimensional Hausdorff measure.
    Furthermore, $\Sigma$ is intrinsically associated to $\mu$. 
\end{rmk}

In the following, we always denote 
\begin{equation}
\pi(\mu)=\Sigma.
\end{equation}
We also define 
\begin{equation}
\Sing(A_\infty)=\{x\in B: \limsup_{r\rightarrow 0} r^{4-2n}
    \int_{B_r(x)}|F_{A_\infty}|^2 >0 \}
\end{equation}
\begin{lem}
The following holds 
\begin{enumerate}
\item $\Sigma=\text{Supp}(\nu)\cup \Sing(A_\infty)$;
\item $\nu$ is absolutely continuous with respect to the $(n-4)$ Hausdorff measure on $\Sigma$. In particular, $\nu=\Theta(x) \mathcal{H}^{n-4}_{\Sigma} $ where 
    $$\epsilon_0^2\leq \Theta(x) \leq C=C(\delta_0, n) 
        \sup_i \|F_{A_i}\|_{L^2(B_{1+\delta_0})}$$
        for $\mathcal{H}^{n-4}$ a.e. $x\in \Sigma$.
\end{enumerate}
\end{lem}

\begin{proof}
For (1), suppose $x\notin \Sigma$, we know $\Theta(\mu, x) < \epsilon_0^2$. By $\epsilon$-regularity, $A_i$ converges smoothly near $x$ which implies $\nu=0$ near $x$ and $A_\infty$ is smooth near $x$. Suppose $x\in \Sigma$, if $x\notin Supp(\nu)$, then 
$$
\lim_{r\rightarrow 0}r^{4-n} \int_{B_r(x)}|F_{A_\infty}|^2=\Theta(\mu, x)\geq \epsilon_0^2.
$$
    i.e.\ $x\in \Sing(A_\infty)$. For (2), by Theorem \ref{Monotonicity} we know 
    that
$$
 r^{4-n}\mu(B_r(x)) \leq \delta_0^{4-n} \mu(B_{\delta_0}(x))
$$
which implies $\mu$ is absolutely continuous with respect to the $(n-4)$-Hausdorff measure. In particular, we have 
$$
\mu|_{\Sigma}=\Theta(x) \mathcal{H}^{n-4}_{\Sigma}. 
$$
for some measurable function $\Theta(x)$. Since 
$$
\lim_{r\rightarrow 0} r^{4-n}\int_{B_x(r)} |F_{A_\infty}|^2\dVol=0
$$ for $\mathcal{H}^{n-4}$ a.e. $x\in \Sigma$, we know 
$$
\nu(x)=\Theta(x) \mathcal{H}^{n-4}_{\Sigma}
$$
for $\mathcal{H}^{n-4}$ a.e. $x\in \Sigma$. The conclusion follows from the density estimate above and the classical fact that 
$$
2^{4-n}\leq \limsup_{r\rightarrow 0} \frac{\Vol_{\H^{n-4}}(\Sigma \cap B_r(x))}{r^{n-4}} \leq 1
$$
for $\H^{n-4}$ a.e. $x\in \Sigma$.
\end{proof}

\subsection{Tangent cone measures}
Fix  $x_0\in B$, define
$$
\tau_\lambda : B_{\delta_0}(x_0)\to B_{\delta_0}(x_0) : x_0+\xi\mapsto
x+\lambda\xi
$$
For $E\subset B_{\delta_0}(x_0)$ measurable, let
$$
\mu_\lambda(E)=\lambda^{4-n}\mu(\tau_\lambda(E))
$$
In this section we prove the following (cf.\ \cite[Lemma 3.2.1]{Tian:00})
\begin{prop}\label{Tangent cone measure}
    For any $\lambda_j\downarrow 0$ there is a Radon measure $\eta$ such
    that (after passing to a subsequence) $\mu_{\lambda_j}\to \eta$ weakly.
    Moreover, $\eta$ is a \emph{cone measure}, in the sense that
    $$
    \lambda^{4-n}\eta(\lambda E)=\eta(E)
    $$
    for any $\lambda>0$ and $E\subset B_{\delta_0}(x_0)$
    measurable.
\end{prop}

\begin{proof}
Let $ds^2_\lambda=\lambda^{-2}\tau_\lambda^\ast ds^2$ be the pull-back
metric and $dV_\lambda$ the associated volume form. Similarly, let
$A_{i,\lambda}=\tau_\lambda^\ast A_i$. We also pull back the hermitian
structure. Then:
$$
    F_{A_{i,\lambda}}=\tau_\lambda^\ast F_{A_i}\quad ;\quad 
|F_{A_{i,\lambda}}|^2(x)=\lambda^4|F_{A_i}|^2(\tau_\lambda(x))
$$
    The weak convergence of $\mu_{\lambda_i}\to \eta$, for some
    Radon measure $\eta$,  follows from
    the monotonicity.
Notice that since
$$
    \sigma^{4-n}\mu(B_\sigma(x_0))\leq \rho^{4-n}\mu(B_\rho(x_0))
$$
we have
$$
    \sigma^{4-n}\eta(B_\sigma(x_0))=\Theta(\mu, x_0)
$$
We wish to show $\eta$ is a cone measure.
For this it suffices to show that for any radially invariant function
    $\phi\geq 0$,
    \begin{equation} \label{eqn:sigma-rho}
        \sigma^{4-n}   \int_{B_\sigma(x)}\phi\, d\eta=
        \rho^{4-n}   \int_{B_\rho(x)}\phi\, d\eta
    \end{equation}
    for all $\sigma$, $\rho$ (cf.\ \cite{Tian:00}, top of p.\ 225).
    By  a diagonalization argument we may assume
    $$
    |F_{A_{i,\lambda}}|^2\, dV_{\lambda_i}\lra \eta
    $$
    weakly.
    To prove \eqref{eqn:sigma-rho}, note that 
    \begin{align}
        \sigma^{4-n}   \int_{B_\sigma(x)}\phi|F_{A_{i,\lambda_i}}|^2\,
        & dV_{\lambda_i}-
        \rho^{4-n}   \int_{B_\rho(x)}\phi|F_{A_{i,\lambda_i}}|^2\,
        dV_{\lambda_i}\notag\\
        &=\int_\sigma^\rho ds\, \frac{d}{ds}\left\{ s^{4-n}\int_{B_s(x)}\phi|F_{A_{i,\lambda_i}}|^2\,
        dV_{\lambda_i} \right\}\notag \\
        &=\int_\sigma^\rho ds\, 
        \frac{d}{ds}\left\{ s^{4-n}\int_{B_1(x)}\phi|F_{\tau_s^\ast A_{i,\lambda_i}}|^2\,
        \tau_s^\ast dV_{\lambda_i} \right\} 
        \label{eqn:phi-integral}
    \end{align}
    Now $s^{4-n}\tau_s^\ast dV_{\lambda_i}=(1+O(s^2\lambda_i))dV_0$, so 
    $$
    \frac{d}{ds}(s^{4-n}\tau_s^\ast dV_{\lambda_i})\lra 0
    $$
    uniformly as $\lambda_i\to 0$.
    Since $F_{A_i}$ has uniformly bounded $L^2$-norm, this term vanishes.
    It suffices to estimate the term coming from 
    $$
    \frac{d}{ds}F_{\tau_s^\ast A_{i,\lambda_i}}=d_{\tau_s^\ast
    A_{i,\lambda_i}}\partial_s(\tau_s^\ast
    A_{i,\lambda_i})
    $$
    At this point we can assume $A_{i,\lambda_i}$ is in radial gauge,
    i.e.\ $\imath_{\partial_r}A_{i,\lambda_i}=0$. Then
    $$
    \imath_{\partial_r}F_{A_{i,\lambda_i}}=\partial_r A_{i,\lambda_i}
    $$
and so
    $$\partial_s(\tau_s^\ast A_{i,\lambda_i})=
   r \imath_{\partial_r}F_{\tau_s^\ast A_{i,\lambda_i}}
    $$
    It follows that
$$
   \frac{d}{ds}(\phi|F_{\tau_s^\ast A_{i,\lambda_i}}|^2)
   =2\langle d_{\tau_s^\ast A_{i,\lambda_i}}(
   r \imath_{\partial_r}F_{\tau_s^\ast A_{i,\lambda_i}}
   ),\phi\,  F_{\tau_s^\ast A_{i,\lambda_i}}\rangle
$$
Integrating by parts, we see that 
\eqref{eqn:phi-integral} is bounded by a constant times the integral of 
$$r^{4-n}| \imath_{\partial_r}F_{ A_{i,\lambda_i}}|
   | F_{ A_{i,\lambda_i}}|
   $$
over $B_\rho(x)$, where the constant depends on $\phi$, $d\phi$, and
$d\Omega$. By Theorem \ref{Monotonicity} we have 
$$
\int_{B_\rho(x)}r^{4-n}| \imath_{\partial_r}F_{ A_{i,\lambda_i}}|^2
dV_{\lambda_i}\lra 0
$$
and so the result follows.

\end{proof}

\begin{rmk}
    An alternative argument follows
    \cite[Lemma 4.1.4]{LinWang:08}. In order to show $\eta$ is a cone
    measure, it suffices to show that for any compactly supported function $\psi$ over $B$ we have 
$$
\frac{d}{ds} (s^{4-n}(\tau_s^*\eta)(\psi))=0.
$$ 
To prove this, note that
$$
\begin{aligned}
\frac{d}{ds} (s^{4-n}(\tau_s^*\eta)(\psi))&=\frac{d}{ds} (s^{4-n}\int_{\R^n} \psi_s d\eta)\\
&=-s^{3-n} \int_{\R^n} ((n-4) \psi_s+s^{-1}x\cdot (\nabla \psi)_s)dV\\
\end{aligned}
$$
where $\psi_s(x)=\psi(x/s)$ 
    and $(\nabla \psi)_s(x)=(\nabla \psi)(x/s)$. So it suffices to show that 
$$
\int_{\R^n} ((n-4) \psi_s+s^{-1}x\cdot (\nabla \psi)_s)dV=0.
$$
From the proof of Theorem \ref{Monotonicity}, we have
$$
\begin{aligned}
&\left|\int_{M} |F_A|^2 (x\cdot \nabla \psi + (n-4) \psi +\psi
    O(r^2))dV\right|\\ \leq
    &\left|-2\sup |d\Omega| \int_{M} \psi r |\iota_{\partial_r} F_A||F_A|dV+
    \int_M 4\psi'r |\iota_{\partial_r} F_A|^2dV+\int_{M} O(r^2)\psi
    |F_A|^2dV\right|.
\end{aligned}
$$
for any $\Omega$-YM  connection $A$ over $(M, g)$ and compactly 
    supported function
    $\psi$. We plug in $(A, \psi)=(A_{i,\lambda_i}, \psi_s)$ and get  
$$
\begin{aligned}
&\left|\int_{\R^n} |F_{A_{i,\lambda_i}}|^2 (s^{-1} x \cdot (\nabla \psi)_s+
    (n-4) \psi_s +\psi_s O(r^2))dV\right|\\
\leq&\left|-2\sup |d\Omega_i| \int_{\R^n} \psi_s r |\iota_{\partial_r}
    F_{A_i^\lambda}||F_{A_i^\lambda}|dV+
    \int_M 4\psi_s'r |\iota_{\partial_r} F_{A_i^\lambda}|^2dV+\int_{\R^n}
    O(r^2)\psi_s |F_{A_i^\lambda}|^2dV\right|
\end{aligned}
$$
By taking limits the right hand side vanishes, and this gives 
$$
\int_{\R^n} ((n-4) \psi_s+s^{-1} x \cdot (\nabla \psi)_s )d\eta=0.
$$
Here, since the base metric converges smoothly to the flat metric on
    $\R^n$, the $O(r^2)$ term vanishes in the limit.
\end{rmk}

Now we fix a tangent measure $\eta$. Define 
$$
L_{\eta}:=\{x\in \R^n: \Theta^{n-4}(\eta, x)=\Theta^{n-4}(\eta, 0)=\Theta^{n-4}(\mu, x_0)\}.
$$
The following can be deduced from the monotonicity formula and the
dimension reduction argument of Federer (cf.\ \cite[p.\ 27]{LinWang:08}).
\begin{lem}
For any $y\in L_{\eta}$, $\eta$ is invariant in the direction of $y$. In particular, $L_\eta$ is a linear subspace of $\R^n$. Furthermore, $\dim L_\eta \leq n-4.$
\end{lem}
Define 
$$
\Sigma_j:=\{x\in \Sigma: \dim L_{\eta} \leq j \text{ for all the tangent measures } \eta \text{ at } x\}.
$$
Then we have 
\begin{prop}
There exists a filtration which consists of closed subsets
$$\Sigma_0 \subset \Sigma_1 \subset \cdots \subset
 \Sigma_{n-4}=\Sigma$$ 
 with the Hausdorff dimension satisfying $\dim(\Sigma_j)\leq j$.
\end{prop}

\subsection{Results parallel to stationary harmonic maps and Yang-Mills connections}
The following \emph{geometric lemma} can be obtained by directly replacing
the energy density associated to the harmonic map with $\Theta^{n-4}$ in
\cite{LinWang:08} or the Yang-Mills case in \cite{Tian:00}
\begin{lem}
Suppose $\Theta^{n-4}(\mu, \cdot)$ is $\mathcal{H}^{n-4}$ approximately
    continuous at $x\in \Sigma$. For any $0<r\ll 1$, there exists $n-4$ points $x^r_1, \cdots x^r_{n-4}$ with
\begin{itemize}
\item $\Theta^{n-4}(\mu, x_i^r) \geq \Theta^{n-4}(\mu, x)-\epsilon_r$ where $\epsilon_r \rightarrow 0$ as $r\rightarrow 0$;
\item $d(x_1, x) \geq r s$ and $d(x_i, x+\text{span}\{x_1-x\cdots, x_{n-4}-x\}) \geq r s$ for some $s\in (0,1)$ independent of $r$.
\end{itemize}
\end{lem}

Given the geometric lemma, we have the existence of \emph{weak tangent planes} as follows
\begin{prop}
For any point $x\in \Sigma$ and any $\delta>0$, there exists $r_x>0$ and a
    tangent plane $L\in \Gr(\mathbb{R}^n, n-4)$ so that $\mu(B_r(x) \setminus L_{\delta r})=0$ 
where $L_{\delta r}$ denotes the $\delta r$ neighborhood  of $L$ in $\mathbb R^n$.
\end{prop}

As a corollary, this implies the \emph{null projection} property.
\begin{prop}\label{NullProjection}
Suppose $E \subset \Sigma$ is a purely $(n-4)$-unrectifiable set, then 
$$
\Vol_{\mathcal{H}^{n-4}}(P_V(E))=0
$$
for any orthogonal projections $P_V: \R^n \rightarrow V\in \Gr(\R^n, n-4)$. 
\end{prop}

\subsection{Positive projection density}
The argument for the following is the same as \cite{LinWang:08} and
\cite{Tian:00}. We will only point out where the change is necessary and refer the reader there for more details.
\begin{prop}\label{PositiveProjection}
For $\mathcal H^{n-4}$ a.e. points $x\in \Sigma$, 
$$
\lim_{r\rightarrow 0} \frac{\Vol_{\mathcal{H}^{n-4}}(P_V(\Sigma\cap B_r(x)))}{\alpha (n-4)r^{n-4}}\geq \frac{1}{2}
$$
for some projection $P_V: \R^n \rightarrow V\in \Gr(\R^n, n-4)$.
\end{prop}

\begin{proof}
Otherwise, we can find a point $x_0\in \Sigma$ so that 
$$
\limsup_{r}r^{4-n}\int_{B_r(x_0)} |F_{A_\infty}|^2=0
$$
and $\Theta^{n-4}(\mu, \cdot)$ is approximately continuous at $x_0\in \Sigma$ but 
$$
\lim_{r\rightarrow 0} \frac{\Vol_{\mathcal{H}^{n-4}}(P_V(\Sigma\cap B_r(x_0)))}{\alpha (n-4)r^{n-4}}<\frac{1}{2}.
$$
In particular, the tangent measure of $\mu$ at $x_0$ takes the form
    $\Theta^{n-4}(x_0) \mathcal{H}^{n-4}_{\R^{n-2}}$ for some
    $\R^{n-2}\subset \R^n$. Recall that from the diagonalization argument
    we assume
$$
    \mu_{\lambda_i} \rightharpoonup \Theta^{n-4}(x_0) \mathcal{H}^{n-4}_{\R^{n-4}}.
$$
Define 
$$
\alpha_{\lambda_i}=\sum^{n-2}_{\alpha=1}
    |\iota_{\partial_\alpha}F_{A_{i,\lambda_i}}|^2\dVol
$$
We know  that for any fixed $\delta>0$ and  $i$ large, 
$
\alpha_{\lambda_i}(B_{3/2}) \leq \delta
$.
Now we define 
    \begin{align*}
        \Fscr_{\lambda_i}& : (\mathbb{R}^{n-4}\times 0) \times (0,1) \rightarrow \R
        \\
        \Fscr_{\lambda_i}(x, \epsilon)&=\int_{B_2^n} |F_{A_{i,\lambda_i}}|^2(x+y)\psi_\epsilon(y_1)\phi^2(y_2) \dVol_y
    \end{align*}
Here, $y=(y_1,y_2)\subset \R^{n-4} \times \R^4$,
    $\psi_\epsilon(y_1)=\epsilon^{4-n}\psi(y_1/\epsilon)$ where
    $\psi$ is a nonnegative compactly supported function on the unit ball
    in $\R^4$ with integral being $1$, while $\phi$ is smooth and compactly supported on the unit ball in $\R^{n-4}$. To simplify the notation, we will denote 
    $F:=F_{A_{i,\lambda_i}}$, $\partial_\alpha=\frac{\partial}{\partial y_\alpha}$ and $\nabla_{\alpha}$ as the covariant derivatives. Viewing $|F|$ as a function of $y$, we have 
$$
\begin{aligned}
\partial_{\alpha} |F|^2&=-2\Tr(\nabla_{\alpha}F_{\gamma \beta} F^{\gamma \beta})\\
&=4\Tr(\nabla_{\gamma} F_{\beta \alpha} F^{\gamma \beta})\\
&=4\partial_{\gamma} \Tr(F_{\beta \alpha} F^{\gamma\beta}(x+y))\pm 4(\iota_{\partial_\alpha} F, *(F\wedge\Omega)).
\end{aligned}
$$
For any $1\leq \alpha\leq n-4$, we have 
$$
\begin{aligned}
    \frac{\partial}{\partial x_\alpha} \Fscr_{\lambda_i}=&\int_{B_2^n} \frac{\partial}{\partial x_\alpha} (|F|^2(x+y))\psi_\epsilon(y_1)\phi^2(y_2)\dVol_y 
\\
=&\int_{B_2^n} \partial_\alpha |F|^2(x+y)\psi_\epsilon(y_1)\phi^2(y_2)\dVol_y \\
=&\int_{B_2^n}  4\partial_{\gamma} \Tr(F_{\beta \alpha} F^{\gamma\beta})(x+y)\psi_\epsilon(y_1)\phi^2(y_2)\dVol_y \\&\pm \int_{B_2^n}4(\iota_{\partial_\alpha} F, *(F\wedge\Omega)))(x+y)\psi_\epsilon(y_1)\phi^2(y_2)\dVol_y\\
=&\sum_{\gamma=n-4}^n\int_{B_2^n}  4\Tr(F_{\beta \alpha} F^{\gamma\beta})\psi_\epsilon(y_1)\frac{\partial}{\partial y_\gamma}\phi^2(y_2)\dVol_y\\
&\pm \int_{B_2^n}4(\iota_{\partial_\alpha} F, *(F\wedge\Omega)))\psi_\epsilon(y_1)\phi^2(y_2)\dVol_y\\&+\sum_{\gamma=1}^{n-4}  4\frac{\partial}{\partial x_\gamma} \int_{B_2^n}  \Tr(F_{\beta \alpha} F^{\gamma\beta})(x+y)\psi_\epsilon(y_1)\phi^2(y_2)\dVol_y
\end{aligned}.
$$
This implies 
$
    \nabla \Fscr_{\lambda_i}=\vec f_{\lambda_i}+ \Div\vec
    G_{\lambda_i}
$,
where 
$$
\begin{aligned}
(\vec f_{\lambda_i})_{\alpha}=&\sum_{\gamma=n-4}^n\int_{B_2^n}  4\Tr(F_{\beta \alpha} F^{\gamma\beta})\psi_\epsilon(y_1)\frac{\partial}{\partial y_\gamma}\phi^2(y_2)\dVol_y\\ 
&\pm \int_{B_2^n}4(\iota_{\partial_\alpha} F, *(F\wedge\Omega)))\psi_\epsilon(y_1)\phi^2(y_2)\dVol_y
\end{aligned}
$$
and 
$$
(\vec G_{\lambda_i})_\alpha^{\gamma}=\int_{B_2^n}  4\Tr(F_{\beta \alpha} F^{\gamma\beta})(x+y)\psi_\epsilon(y_1)\phi^2(y_2)\dVol_y.
$$
Here the divergence of $\vec G_{\lambda_i}$ is taken for each vector component of 
$\vec G_{\lambda_i}$. Since $\alpha_{\lambda_i} \rightharpoonup 0$,
we know that for any $\delta>0$, 
$$
\|\vec f_{\lambda_i}\|_{L^2(B_2^{n-4})}+\|\vec G_{\lambda_i}\|_{L^2(B_2^{n-4})} \leq \delta
$$
for $i$ sufficient large and $\lambda$ sufficiently small.  Given this, by
\cite[Lemma 4.2.10]{LinWang:08} 
we know for any $\delta_1$ there exist constants $C_{\lambda_i}(\epsilon)$
$$
\|\Fscr_{\lambda_i}(\cdot, \epsilon)-C_{\lambda_i}(\epsilon)\|_{L^1(B^{n-2}_{2})} \leq \delta_1.
$$
Letting $\epsilon \rightarrow 0$, we have for some constants
$C_{\lambda_i}$,
$$
\left|\int_{B_2^{n-4}}|F_{A_{i,\lambda_i}}|^2(a,y_2)\phi^2(y_2)
dy_2-C_i^\lambda\right|\leq \delta_1
$$
when $i$ large. As in \cite{Lin:99, Tian:00}, this then implies 
$
\lim C_{\lambda_i}=\Theta^{n-4}(\mu, x_0)
$.
It then follows as in those references that the projection from $\R^n \rightarrow \R^{n-4}
\times 0$ will give a contradiction.
\end{proof}

\subsection{Proof of Theorem \ref{thm:rectifiable}}
Now we are ready to finish the proof for Theorem \ref{thm:rectifiable} as in
\cite{Lin:99, Tian:00}. 
By the Besicovitch-Federer decomposition theorem, 
we can write  $\Sigma=\Sigma^{r} \cup \Sigma^{u}$,
where $\Sigma^{r}$ is $(n-4)$-rectifiable while $\Sigma^u$ is purely
$(n-4)$-unrectifiable. Furthermore,  if $\Sigma^u\neq\emptyset$, then
$\Vol_{\mathcal{H}^{n-4}}(\Sigma^u)>0$. By Proposition \ref{NullProjection}, we know
$$
\Vol_{\mathcal{H}^{n-4}}(P_V(\Sigma^u\cap B_r(x)))=0
$$
while by Proposition \ref{PositiveProjection}, we have 
$$
\Vol_{\mathcal{H}^{n-4}}(P_V(\Sigma^u\cap B_r(x)))>0
$$ 
for $0<r\ll 1$. This is a contradiction. In particular, this implies 
$\Vol_{\mathcal{H}^{n-4}}(\Sigma^u)=0$, and so $\Sigma^u=\emptyset$. 
Thus, $\Sigma$ is $(n-4)$-rectifiable.

\section{Weak compactification of the moduli space of smooth $\Omega$-Yang-Mills connections}\label{ModuliCompactness}
In this section, we will study the compactification of the moduli space of smooth $\Omega$-YM connections on a fixed bundle $E$ with bounded $L^2$ norm of curvature over $(M,g)$. We denote the moduli space as 
$$
\mathcal A_{\Omega, c}:=\{A \in\mathcal{A}: d_A^*(F_A+*(F_A\wedge \Omega))=0, \int_{M} |F_A|^2 \leq c\}
$$
Given a sequence $A_i\in \mathcal \mathcal A_{\Omega, c}$, by passing to a
subsequence, we can assume $|F_{A_i}|^2\dVol$ converges to $\mu$ a sequence
of Radon measures,  and
modulo gauge transformations, $A_i$ converges to $A$ outside $\pi(\mu)$. Define $\overline{\mathcal A_{\Omega, c}}$ to be the space of such pairs $(A, \mu)$.

\begin{defi}
Given a sequence $(A_i, \mu_i) \in \overline{\mathcal A_{\Omega, c}}$, we
    say $A_i$ converges to a finite energy $\Omega$-YM connection $(A_\infty, \mu_\infty)$ if 
\begin{enumerate}
\item $\mu_i$ converges to $\mu_\infty$ weakly as a sequence of Radon measures;
\item up to gauge transforms, $A_i$ converges to $A_\infty$ outside $\pi(\mu_\infty)$.
\end{enumerate}
\end{defi}
\begin{thm}\label{Sequential Compactness}
$\overline{\mathcal A_{\Omega, c}}$ is weakly sequentially compact in the sense that every sequence $\{(A_i, \mu_i)\}$ in $\overline{\mathcal A_{\Omega, c}}$ sub-converges to some $(A_\infty, \mu_\infty)\in \overline{\mathcal A_{\Omega, c}}$.
\end{thm}
\begin{proof}
 Given a sequence $(A_i, \mu_i) \in \overline{\mathcal A_{\Omega, c}}$, by assumption, for each $i$, we can find a sequence of $\{A_{ij}\}_j$ so that $\mu_{ij}=|F_{A_{ij}}|^2\dVol$ converges to $\mu_i$ weakly as a sequence of Radon measures. By a diagonal sequence argument, we can assume $\mu_{ij}$ and $\mu_i$ both converge weakly to $\mu_\infty$ as sequences of Radon measures. The following now is needed to guarantee the existence of the limit of $A_{i}$
\begin{equation}
\limsup_i \pi(\mu_i) \subset \pi(\mu_\infty).
\end{equation} 
Suppose this is not true. By passing to a subsequence, there exists a sequence of points $x_i \in \pi(\mu_i)$ which converges to $x_\infty \notin \pi(\mu_\infty)$. In particular, we have for $0<r<dist(x_\infty, \pi(\mu_\infty))$
$$
\mu_\infty(\partial B_r(x_\infty))=0,
$$ 
which implies 
$
r^{4-n}\mu_i(B_r(x_i))\leq \epsilon_0/2
$,
for $r$ sufficiently small. This, of course,
    contradicts with the assumption that $x_i \in \pi(\mu_i)$. Given this,
    up to gauge transforms, we can assume $A_i$ sub-converges to $A_\infty$
    outside $\pi(\mu_\infty)$ smoothly. Indeed, a priori, we only know that
    $A_i$ converges to $A_\infty$ outside a closed subset $\widetilde{\Sigma} \subset M \setminus \pi(\mu_\infty)$ of  Hausdorff codimension at $4$ set. However, since we already know that $\mu_\infty|_{M\setminus \pi(\mu_\infty)}=|F_{A_\infty}|^2\dVol$, by Lemma \ref{Lem4.2}, we know 
$$
r^{4-n}\mu_i(B_r(x))\leq \epsilon_0/2 
$$
for $i$ large. This implies that $A_i$ converges to $A_\infty$ smoothly
    over $B_r(x)$. In particular, we know $\widetilde{\Sigma}=\emptyset$, i.e.\ $A_i$ sub-converges to $A_\infty$ smoothly outside $\pi(\mu_\infty)$. Now by a diagonal sequence argument again, we can assume $A_{ij}$ sub-converges to $A_\infty$ smoothly outside $\pi(\mu_\infty)$. The sequential compactness follows.
\end{proof}
\begin{rmk}
\begin{itemize}
\item For general finite energy $\Omega$-YM connections on a fixed bundle over $M$, or even YM connections, we do not know  whether we can take a limit or not due to lack of control of $\text{Sing}(A_i)$. It is very crucial to assume they all come from limits of smooth connections here. 
\item The compactness we obtain here is very weak due to the fact that the
    limiting bundles $E_\infty$ are not known to be isometric to
        $E|_{M\setminus \Sigma}$. This does, however, hold in the case of Hermitian-Yang-Mills connections over general complex manifolds (see Corollary \ref{LimitingBundleOfHE})
\end{itemize}
\end{rmk}

\section{Singularity formation}\label{SingularityFormation}
\subsection{Bubbling connections at a generic point}
Using the proof of Proposition \ref{PositiveProjection}, the argument in \cite[Prop.\
4.1.1]{Tian:00} for 
the case of Yang-Mills connections   gives 
\begin{prop}
Fix a point  $x\in \Sigma$ so that
\begin{itemize}
\item the tangent plane of $\Sigma$ at $x$ exists uniquely;
\item $\Theta^{n-4}(\mu, \cdot)$ is $\mathcal{H}^{n-4}$-Hausdorff continuous at $x$ ;
\item $\limsup_r r^{4-n} \int_{B_r}|F_{A_\infty}|^2=0$.
\end{itemize} 
By passing to a subsequence, up to gauge transforms, $A_{i,\lambda_i}$ converges to a  $\Omega_x$-YM connection $B_\infty$ over $\R^n$  with $\R^n=T_x\Sigma \times (T_x\Sigma)^\perp$ satisfying 
$
\iota_v F_{B_\infty}=0
$,
for any $v\in T_x\Sigma$. 
\end{prop}
Following \cite{Tian:00}, we call $B_\infty$  a \emph{bubbling connection} of the sequence $\{A_i\}$ at $x$.
\subsection{Tangent cones of the limits}\label{Tangent cone}
Denote $(A_\infty^\lambda, \mu_\infty^\lambda)=\lambda^*(A_\infty, \mu_\infty)$ where $\lambda: B_{\lambda^{-1}\delta_0}(x) \rightarrow B_{\delta_0}(x)$. 
\begin{prop}
By passing to a subsequence, 
\begin{itemize}
\item $\mu_\infty^\lambda$ converges to a cone measure $\eta$;
\item up to gauge transforms, $A_\infty^\lambda$ converges to $A_\infty^{c}$ outside 
$$\pi(\eta)=\{x\in \R^n: \Theta^{n-4}(\eta, x) \geq \epsilon_0^2\}$$
which is scaling  invariant. Furthermore, $\iota_{\partial_r} F_{A_\infty^c}=0.$
\end{itemize}
\end{prop}
\begin{proof}
The first statement follows from Proposition \ref{Tangent cone measure}. Given this, it follows the same as Theorem \ref{Sequential Compactness} that 
$$
\limsup_{\lambda} \pi(\mu_\infty^\lambda) \subset \pi(\eta).
$$
Now up to gauge transforms, we can assume $A_\infty^\lambda$   sub-converges to $A_\infty^c$ smoothly outside $\pi(\eta)$. It follows from the monotonicity formula that 
$
\iota_{\partial_r} F_{A_\infty^c}=0
$,
outside $\pi(\eta)$. Since $\eta$ is a cone measure, we know also $\pi(\eta)$ is also a cone.  
\end{proof}
We call $(A_\infty^c, \eta)$ a \emph{tangent cone} of $(A_\infty, \mu_\infty)$ at the point $x$. A priori, we donot know whether it is unique or not since this involves a choice of the subsequence. 

\begin{rmk}
In \cite{Tian:00}, 
    the tangent cones of general stationary Yang-Mills connections are shown to exist where the stationary condition is needed for the monotonicity formula. Here as long as we know $(A_\infty,\mu_\infty)$ comes from the limit of smooth connections, it already has a monotonicity property that suffices for use. 
\end{rmk}

\subsection{$\Omega$-ASD instantons and calibrated geometries}
Given the analytic results above, it is straightforward to see that the results in
\cite{Tian:00} hold for general $\Omega$-ASD instantons without assuming
$\Omega$ to be closed. More precisely, we assume $(A_\infty, \mu_\infty)$
is an finite energy $\Omega$-ASD instanton which comes from the limit of a sequence of smooth $\Omega$-ASD instantons with uniformly bounded $L^2$ norm on curvature. We also write 
$$
\mu_\infty=|F_{A_\infty}|^2 \dVol+\Theta^{n-4}(x) \mathcal{H}^{n-4}_{\Sigma}
$$
as before. Similar to Proposition $4.2.1$ in \cite{Tian:00}, the following holds
\begin{prop}
A bubbling connection $B_\infty$ of $(A_\infty, \mu_\infty)$ at $\mathcal{H}^{n-4}$ a.e.  $x\in \Sigma$ is a $\Omega_x$-ASD instanton. In particular, $\Omega_x$ induces a volume form of $\Sigma$ at $x$.
\end{prop}
This 
implies the following, as
pointed out  in the Yang-Mills case in 
\cite[p.\ 242, Remark 5]{Tian:00}). The proof is exactly the same.
\begin{thm}\label{Calibrated}
For the limiting connection $(A_\infty, \mu_\infty)$
\begin{itemize}
\item $\displaystyle\frac{1}{8\pi^2}\Theta^{n-4}(x)$ is integer valued at
    $\mathcal{H}^{n-4}$ a.e.\ $x\in \Sigma$; \\
\item  $\Omega$ restricts to a volume form of $T_x\Sigma$ at $\mathcal H^{n-4}$ a.e. $x\in \Sigma$.
\end{itemize}
\end{thm}

\section{Removable Singularities}\label{RemovableSingularities}
In this section, using the main results in \cite{SmithUhlenbeck:18}
we generalize the removable singularity theorem for
stationary Yang-Mills fields in \cite{TaoTian:04} 
to  the case of  $\Omega$-YM connections. The argument  closely follows
\cite[Theorem 10]{SmithUhlenbeck:18}. Below we will denote by $A$  an $\Omega$-YM connection
defined on the trivial bundle over $M\setminus \Sigma$,  where $M=[-4,4]^{n}$
endowed with a smooth Riemannian metric, $\Omega$ is a smooth $(n-4)$-form
on $M$, and $\Sigma$
is a closed subset of $U$ of finite $(n-4)$-dimensional Hausdorff measure.

\begin{thm}
If $\sup_{x\in M} \sup_{\sigma} f_2(x, r)$ is sufficiently small, then for
    any $B_{r}(x)\subset \Omega$, there exists a gauge transform $g$ over
    $B_{r}(x)\setminus \Sigma$ so that $g(A)$ extends to 
    a smooth connection over $B_r(x)$.
\end{thm}

\begin{proof}
Denote $f=|F_{A}|$. It suffices to show that $f$ satisfies
\begin{equation}\label{SU-type inequality}
-\Delta f+\alpha \frac{|df|^2}{f}-c|F_A|^2 f\leq C f
\end{equation}
    over $M\setminus \Sigma$ for some $\alpha>0$. Indeed, given
    \eqref{SU-type inequality}, by 
    \cite[Thm.\ 9]{SmithUhlenbeck:18} we know that
    $f \in L^\infty ([-1,1]^{n})$. Now the existence of the gauge
    transformation  follows from 
    \cite[App.\ C, Thm.\ 19]{SmithUhlenbeck:18}.
    It remains to show that $f$ satisfies the inequality (\ref{SU-type
    inequality}). By  \eqref{eqn:weitzenbock} we have 
$$
\begin{aligned}
-\frac{1}{2}\Delta |F_A|^2&=-|\nabla_A F_A|^2+(\nabla_A^*\nabla_A F_A, F_A)\\
&=-|\nabla_A F_A|^2+(\{F_A, F_A\}, F_A)+(\{R_g, F_A\}, F_A)+(\{d\Omega, \nabla_A F_A\}, F_A) 
\end{aligned}
$$
which implies 
$$
\begin{aligned}
&-\frac{1}{2}\Delta |F_A|^2+|\nabla_A F_A|^2+|d_A F_A|^2+|d_A^*F_A|^2\\
\leq& (\{F_A, F_A\}, F_A)+(\{R_g, F_A\}, F_A)+(\{d\Omega, \nabla_A F_A\}, F_A)+|d \Omega \wedge F_A|^2\\
\leq & C |F_A|^3 + C_\epsilon |F_A|^2+\epsilon |\nabla_A F_A|^2
\end{aligned}
$$
where the last line follows from H\"older's inequality,  and
    $0<\epsilon\ll 1$ is to be determined later. This then implies 
    \begin{equation} \label{eqn:ineq}
-\frac{1}{2}\Delta |F_A|^2+(1-\epsilon)(|\nabla_A F_A|^2+|d_A F_A|^2+|d_A^*F_A|^2) -C |F_A|^3 \leq  C_\epsilon |F_A|^2.
    \end{equation}
Now the improved Kato inequality (see \cite[Thm.\ 5]{SmithUhlenbeck:18})
    gives 
$$
|\nabla_A F_A|^2+|d_A F_A|^2+|d_A^*F_A|^2 \geq \frac{n}{n-1} |d|F_A||^2.
$$
    Combined with \eqref{eqn:ineq} this gives
$$
-\frac{1}{2}\Delta |F_A|^2+(1-\epsilon) \frac{n}{n-1} |d|F_A||^2-C |F_A|^3 \leq  C_\epsilon |F_A|^2.
$$
Substituting $f=|F_A|$ and $u=|F_A|^2$, we have
$$
-\frac{1}{2}\Delta f^2+\frac{(1-\epsilon)n}{n-1} |df^2|^2-Cuf\leq C_\epsilon f^2.
$$
A straightforward calculation now shows 
$$
-\Delta f + \left(\frac{(1-\epsilon)n}{n-1} -1\right)
    \left|\frac{df}{f}\right|^2-C u \leq C_\epsilon f.
$$
Choose $\epsilon$ so that $\displaystyle\alpha=\frac{(1-\epsilon)n}{n-1} -1 > 0$, and
    \eqref{SU-type inequality} follows. 
\end{proof}

\section{Hermitian-Yang-Mills connections over general complex manifolds}
\subsection{Improvement of the analytic results}
In this section, we will generalize Tian's holomorphic cycle theorem for
Hermitian-Yang-Mills connections over
K\"ahler manifolds  \cite[Thm.\ 4.3.3]{Tian:00}
to the case of Hermitian manifolds. More precisely, we fix $A_i$ to be a
sequence of HYM connections over an $m$-dimensional Hermitian manifold $(X, \omega)$ with $\|F_{A_i}\|\leq
C$. These are \emph{not} Yang-Mills connections in general. As before, let 
$$
\Sigma=\{x\in B: \lim_{r\rightarrow 0^{+}}\liminf_{i} r^{4-2m} \int_{B_x(r)} |F_A|^2 \geq \epsilon_0^2\}.
$$
Then we can assume 
\begin{itemize}
\item $\mu_i:=|F_{A_i}|^2\dVol \rightharpoonup
    \mu=|F_{A_\infty}|^2\dVol+\nu$ where $\text{supp}(\nu)$ is equal to the
        pure complex codimension $2$ part of $\Sigma$;
\item up to gauge transforms, $A_i$ sub-converges to $A_\infty$ outside $\Sigma$.
\end{itemize}

\begin{rmk}
Strictly speaking, without assuming the Hermitian-Einstein constant vanishes, i.e. $\sqrt{-1}\Lambda
    F_{A}=0$, HYM connections are not exactly $\Omega$-ASD instantons  in
    the  sense of \eqref{eqn:asd},  where $\Omega=\omega^{m-2}/(m-2)!$. But it is projectively $\Omega$-ASD connections in the sense that 
    $$*(F_A^{\perp}\wedge \Omega)=-F_A^{\perp}$$
     where $F_{A}^\perp=F_A-\mu \Id\omega$ satisfying $F_{A}^\perp \wedge
     \omega^{m-1}=0$. It is straightforward to see that the results for
     $\Omega$-YM connections holds for this case by using the same
     argument. There is another way to see this. By the Bochner-Kodaira-Nakano identity (see \cite[Theorem 1.1]{Demailly}), we have 
$$
d_A^*F_A=\rho F_A
$$
for some $\rho=\rho([\Lambda,\partial\omega], [\Lambda,\db \omega])$, for which the 
   same arguments as for $\Omega$-YM connections  apply. The results
    in the previous sections hold in this case.
\end{rmk}
The following can be deduced easily from  \cite[Thm.\ 2]{BandoSiu:94}.
\begin{prop}\label{extension}
\begin{enumerate}
\item  $E_\infty$ can be extended uniquely as a reflexive sheaf $\E_\infty$ over $M$. For any local section $s\in \E_\infty$, $\log^+|s|^2\in H^1_{loc}\cap L^\infty_{loc}$. Furthermore, $A_\infty$ can be extended to be defined over $M\setminus \Sing(\E_\infty).$ In particular, $\Tr(F_{A_\infty} \wedge F_{A_\infty})$ is closed across $\Sigma$, thus the current 
$$
c_2(\Sigma)=\lim_{j_i} \Tr(F_{A_{j_i}}\wedge F_{A_{j_i}})-\Tr(F_{A_\infty}\wedge F_{A_{\infty}})
$$
is closed.
\item $
\Sigma=\Sing(\E_\infty) \cup \cup_k \Sigma_k
$ is a complex subvariety of $M$ and  \begin{equation} \label{eqn:c2}
    c_2(\Sigma)=\sum m_k [\Sigma_k].
\end{equation} 
In particular, $\nu=\sum m_k \mathcal H^{2n-4}_{\Sigma_k}$ where $\Sigma_k$ are the irreducible pure codimension $2$ components of $\Sigma$ and 
\begin{equation}\label{Eqn7.4}
\mu_\infty=|F_{A_\infty}|^2\dVol+\sum_k m_k \mathcal{H}^{2n-4}_{\Sigma_k}.
\end{equation}
\end{enumerate}
\end{prop}

\begin{proof}
For (1), locally by replacing $\omega$ with any K\"ahler metric, it does
    not change the fact that $\|F_{A_\infty}\|_{L^2_{loc}}<\infty$ . By Theorem $2$ in
    \cite{BandoSiu:94}, 
    we know that $E_\infty$ can be extended 
    uniquely as a reflexive sheaf $\E_\infty$ over $M$. Furthermore, for
    any local section $s\in \E_\infty$, $\log^+|s|^2\in H^1_{loc}$. Then
    the local $L^\infty$ bound follows from Moser iteration. Given this, one can directly
    repeat the proof for Proposition $1$ in \cite{BandoSiu:94} to extend
    $A_\infty$ by extending the metric $H_\infty$ locally. Now we use
    Simpson's trick to show the closedness of $\Tr(F_{A_\infty} \wedge
    F_{A_\infty})$ (see \cite[p.\ 71]{Simpson:87}). 
    By proceeding with stratum of $\Sing(\E_\infty)$ which has codimension
    at least $6$, we can choose a point $x\in \Sing(\E_\infty)$ which is
    smooth at $x\in \Sing(\E_\infty)$. Let $\psi$ be a smooth $(n-5)$-form
    which is compactly supported near $x$.  
    \begin{itemize}
    \item Suppose $\psi$ has vanishing constant coefficients. We can choose a family of cut-off function $\phi_\epsilon$ which vanishes over an $\epsilon$-neighborhood of $x$ and $d(\phi_\epsilon\psi)$ is uniformly bounded. In particular, we have 
    $$
\int_{M} \Tr(F_{A_\infty} \wedge F_{A_\infty}) \wedge d\psi  =\lim_{\epsilon\rightarrow 0}\int_{M} \Tr(F_{A_\infty} \wedge F_{A_\infty}) \wedge d(\phi_\epsilon \psi)=0.    
    $$
    \item In general, since $\Sing(\E_\infty)$ has codimension at least $6$, we know that 
    $
\psi=\sum_i dx_i\wedge \omega_i    
    $,
    where $x_i$ are defining coordinates for $\Sing(\E_\infty)$. Now $\psi-\sum_id(x_i \omega_i)$ vanishes along $\Sing(\E_\infty)$ and satisfies $d(\psi-\sum_id(x_i \omega_i))=d\psi$. By the special case above, we know 
    $$\int_{M} \Tr(F_{A_\infty} \wedge F_{A_\infty}) \wedge d\psi=0.$$
\end{itemize}     
Now we prove $(2)$. We first show $\Sing(\E_\infty) \cup \cup_k\Sigma_k \subset \Sigma$. From the above, we know
    $\Sing(A_\infty)\subset\Sing(\E_\infty)$. It remains to
    show that $\text{Supp}(\nu)$ is a pure codimension $2$ subvariety of
    $M$. Indeed, we know $\Sigma$ is calibrated by $\omega^{m-2}/(m-2)!$, which
    implies $T_x\Sigma$ is a complex analytic subspace of $T_xM$.  Given
    this, it follows from part (1) and Theorem \ref{Calibrated}
    that $c_2(\Sigma)$ is a closed integral current. Then by 
     King's theorem  \cite{King:71} we can express
     $c_2(\Sigma)$ in the form \eqref{eqn:c2}
for some integers $m_k$ and pure codimension $2$ subvarieties $\Sigma_k$ of $M$. 
    This implies 
$
\Sigma \subset \Sing(\E_\infty) \cup \cup_k\Sigma_k.
$,
through which the top pure codimension $2$ parts are identified. For the
    other direction, suppose not, there exists a point $x\in
    \Sing(\E_\infty)$ with $\Theta^{n-4}(\mu_\infty, x)=0$. As Theorem
    \ref{Sequential Compactness}, we can conclude that
    $r^{4-2n}\mu_i(B_r(x))<\epsilon_0/2$, for $i$ large and $r$ small. This implies that $A_i$ sub-converges to $A_\infty$ smoothly near $x$, which gives a contradiction. In sum, we have 
$
\Sigma = \Sing(\E_\infty) \cup \cup_k\Sigma_k
$.
\end{proof}

\begin{rmk}
\begin{itemize}
    \item It follows by exactly the same argument that Proposition
        \ref{extension} (1) holds for general admissible
        Hermitian-Yang-Mills connections over complex Hermitian manifolds,
        i.e.\ smooth Hermitian-Yang-Mills connections defined away from a closed Hausdorff codimension $4$ set.

\item It is straightforward to see that the proof for the closedness part
    holds for general finite energy $\Omega$-YM connections with mild
        singularities; for example, when the singular set can be stratified
        by smooth manifolds of real codimension at least $6$. In general,
        it is conjectured that the set of essential singularities of finite
        energy $\Omega$-ASD instantons when $\Omega$ is closed has Hausdorff codimension at least $6$ (see \cite{Tian:00}).
\end{itemize}
\end{rmk}

\begin{cor}\label{LimitingBundleOfHE}
As a smooth bundle, $E_\infty|_{M\setminus \Sigma}\cong E|_{M\setminus
    \Sigma}$. In particular, we can assume there exists a sequence of
    bundle isometries $\Phi_{j_i}: E_{\infty} \rightarrow E|_{M\setminus
    \Sigma}$ so that $\Phi_{j_i}^*A_{j_i}$ locally converges to $A_\infty$ smoothly away from $\Sigma$.
\end{cor}

Given this, let $E$ be a Hermitian bundle over a compact Hermitian manifold $(M, \omega)$. Denote $\overline{M_{HYM, c}}$ to be the space of limits of smooth Hermitian-Yang-Mills connections on $E$ with $L^2$ norm of curvature bounded by $c$ \emph{mod gauge} (smooth wherever the connections are smooth). We give $\overline{M_{HYM,c}}$ a topology by specifying a basis of open neighborhood as $\mathcal{U}_{\vec{\epsilon},\phi}([A, \mu])$ consisting of $[(A', \mu')]\in \overline{M_{HYM, c}}$ satisfying 
\begin{itemize}
\item $A'$ lies in the $\epsilon_1$ neighborhood of $A$ outside a $\epsilon_1$ neighborhood of $\pi(\mu)$;
\item $|\mu(\phi)-\mu'(\phi)|<\epsilon_2$.
\end{itemize}
Here $\vec{\epsilon}=(\epsilon_1, \epsilon_2)$ with $\epsilon_i>0$ for $i=1,2$ and $\phi$ is a continuous and bounded function. 
\begin{rmk}
When $m=2$, this topology coincides exactly with the topology in the case of four dimensional manifolds (see \cite[Section $4.4$]{DonaldsonKronheimer:90}).
\end{rmk}
Given this,  we have the following improved version of Theorem \ref{Sequential Compactness}
\begin{thm}\label{Theorem: c-HYM}
$\overline{M_{HYM,c}}$ is a first countable sequentially compact Hausdorff space.
\end{thm}
By Proposition \ref{extension}, the moduli space can be also viewed as consisting of pairs $(A_\infty, \mathcal{C}^{an})$ mod gauge where $\mathcal{C}^{an}=\sum_k m_k \Sigma_k$ is a integer linear combination of pure codimension two subvarities of $X$. Later we will not make a difference between them. 

\subsection{HYM connections over a class of balanced manifolds of Hodge-Riemann type}\label{Section:Hodge-Riemann} 
Now we assume $(M,\omega)$ is an $m$-dimensional compact balanced
Hermitian manifold of Hodge-Riemann type as defined in \cite[Def.\ 2.7]{ChenWentworth:21b}. This means we can write
$$
\omega^{m-1}=\omega_0\wedge\Omega_0
$$
where $\omega_0$ is a strictly positive $(1,1)$ form, $\Omega_0$ is of type
$(m-2,m-2)$, and
\begin{enumerate}
\item $d\omega^{m-1}=0$;
\item $d\Omega_0=0$;
\item for any $p+q=2$, there exists a pointwise $Q$-orthogonal decomposition 
$$\Lambda^{p,q}=\C \omega_0\oplus P^{p,q}$$
where $P^{p,q}=\{\alpha\in \Lambda^{p,q}: \alpha \wedge \omega_0 \wedge \Omega_0= 0\};$
\item  $Q(\alpha, \beta):=(\sqrt{-1})^{p-q} (-1)^{\frac{(p+q)(p+q-1)}{2}}
*(\alpha \wedge \overline{\beta}\wedge \Omega_0)$ is positive definite on $P^{p,q}$. 
\end{enumerate}
In this case, a uniform bound 
for the $L^2$ norm of curvature of all the smooth irreducible 
Hermitian-Yang-Mills connections is automatic by the following observation.
\begin{lem}
Given any HYM connection $A$ on $E$, 
$$
\int_{X} |F_A|^2\frac{\omega^{m}}{m!}\leq C
$$
where $C=C(c(E), \omega_i)$.
\end{lem}
\begin{proof}
    By conditions (3) and (4) we have 
$$
\int_{X} |F_A|^2\frac{\omega^{m-1}}{(m-1)!}\leq C_1 (\int_X \Tr( F_A\wedge F_A)\wedge \Omega_0 + C_2\int_{X} |f|^2\omega_0\wedge \omega_0 \wedge \Omega_0 )
$$
where $F_A^\perp=F_A-f \Id\omega_0$. Here 
    $$f=\mu \frac{\frac{\omega^n}{n!}}{\omega_0\wedge \omega_0 \wedge \Omega_0}$$ In particular,  we have 
\begin{equation}
\int_{X} |F_A|^2\frac{\omega^{n}}{n!}\leq C_1(\int_X F_A \wedge F_A \wedge \Omega_0 + C_2 \mu^2 \int_X \frac{\frac{\omega^n}{n!}}{\omega_0\wedge \omega_0 \wedge \Omega_0} \frac{\omega^n}{n!}).
\end{equation}
The result follows.
\end{proof}
 In this case, we denote the compactification of the moduli space of HYM connections mod gauge as $\overline{M_{HYM}}$ by choosing $c$ large.
\begin{thm}\label{HEKahler}
On a unitrary bundle over a compact balanced Hermitian manifold $(X,\omega)$ of Hodge-Riemann type, $\overline{M_{HYM}}$ is a first countable sequentially compact Hausdorff space.
\end{thm}
Now we would like to give an  important class of balanced metrics of
Hodge-Riemann type, which comes from \emph{multipolarizations}. Namely, for any positive $(1,1)$ forms $\omega_0, \cdots, \omega_{m-2}$ on a compact complex manifold $X$ so that 
\begin{align}
    \begin{split} \label{eqn:multiporizations}
        \frac{\omega^{m-1}}{(m-1)!}&=\omega_0\wedge \cdots\wedge
 \omega_{m-1}\\
        d(\omega_0\wedge \omega_1 \wedge \cdots\wedge \omega_{m-2})&=0\\
        d(\omega_1 \wedge \cdots\wedge \omega_{m-2})&=0
    \end{split}
\end{align}
then 
by the main result in \cite{Timorin:98}
we get a balanced Hermitian metric $\omega$ of Hodge-Riemann type 
by setting $\Omega_0=\omega_1 \wedge \cdots\wedge \omega_{m-2}$.
\begin{cor}
On a unitrary bundle over a compact balanced Hermitian manifold $(X,\omega)$ satisfying (\ref{eqn:multiporizations}), $\overline{M_{HYM}}$ is a first countable sequentially compact Hausdorff space.
\end{cor}
In particular, this gives the following 
\begin{cor}
On a unitrary bundle over a compact K\"ahler manifold $(X,\omega)$, $\overline{M_{HYM}}$ is a first countable sequentially compact Hausdorff space.
\end{cor}
\begin{rmk}
When $(X,\omega)$ is a projective algebraic manifold, i.e. $\omega=c_1[L]$ for some line bundle $L$, 
    it is known that $\overline{M_{HYM}^*}$, which denotes the closure of the space of irreducible HYM connections with \emph{fixed determinants} in $\overline{M_{HYM}}$, admits a complex structure coming from the algebraic geometric side. The induced complex structure makes it an algebraic space (see \cite{GSTW:18}). We will explain how it can be generalized to the case of multipolarizations in the following by using the same argument in \cite{GSTW:18} and the algebraic geometric results in \cite{GrebToma:17}. 
\end{rmk}

\subsection{$\overline{M_{HYM}^*}$ for multipolarizations}\label{Section:multipolarization}
In this section, we fix $(E,H)$ to be a unitary vector bundle over a compact complex Hermitian manifold $(X, \omega)$ so that 
$$
\frac{\omega^{m-1}}{(m-1)!}=\omega_0 \wedge \cdots\wedge \omega_{m-2}
$$
where $[\omega_i]$ are all ample classes, i.e. $[\omega_i]=c_1(L_i)$ for
some ample line bundles $L_i$. Set $\Omega_0=\omega_1\wedge \cdots\wedge \omega_{m-2}$. 
As mentioned above, we can view the moduli space $\overline{M_{HYM}^*}$ consisting of pairs $(A_\infty, \mathcal{C}^{an})$ mod gauge. It is a sequentially compact Hausdorff space. Using the argument in \cite{GSTW:18}, we briefly explain how a complex structure could be given to $\overline{M_{HYM}^*}$ to make it an algebraic space.

\subsubsection{Moduli space of semistable torsion free sheaves via multipolarizations}
In this section, we will recall the construction for the compactification of the moduli space of semistable sheaves with given numerical classes and fixed determinant. We refer the readers to \cite{GrebToma:17} for more details. Recall that the space of slope semistable sheaves having the same Chern classes as $E$ over $(X,\omega)$ is bounded, i.e. if we fix $\O(1)$ to be any polarization of $X$, for fixed $k$ large enough, for any $\E$, we have 
$
H^i(X, \E(k))=0
$,
for $i>1$, and $\E(k)$ is globally generated. Let 
$$
\mathcal H=\C^{\oplus \tau(k)} \otimes \O(-k)
$$
where $\tau$ denotes the Hilbert polynomial of $\E$. Now we know for $k$ fixed large enough, all such sheaves can be viewed as points $[q: \mathcal H \rightarrow \E]$ in $Quot(\mathcal H, \tau)$ by choosing an isomorphism $\C^{\oplus \tau(k)} \cong H^0(X,\E(k))$. Here $Quot(\mathcal H, \tau)$ denotes the space of points given by surjective maps
$
q: \mathcal H \rightarrow \E
$,
where the Hilbert polynomial of $\E$ is equal to $\tau_E$,  modulo the equivalence: $q: \mathcal H \rightarrow \E$ and $q': \mathcal H \rightarrow \E'$ are equivalent if and only if there exists an isomorphism $f\circ q=q'$, i.e. $\ker(q)=\ker(q')$.
Furthermore, there exists a universal quotient 
$$
q_{\mathcal U}: \mathcal O_{Quot(H,\tau_E)} \otimes \mathcal H \rightarrow \mathcal U.
$$
over $Quot(H, \tau_E) \times X$ which restricts to the natural quotient at each point $[q]$. Now we denote $R^{\mu ss}$ as the subscheme of $Quot(\E,\mathcal{H})$ consisting of elements $[q: \mathcal H \rightarrow \E]$ so that 
\begin{itemize}
\item $\E$ is semistable;
\item $\text{det} (\E)=\mathcal{J}$;
\item $\E$ has the same numerical classes as $\E$;
\item $q$ induces an isomorphism between $\C^{\oplus \tau(k)}$ and $H^0(X, \E(k))$.
\end{itemize}
Define $\mathcal{Z}$ as the weak normalization of the reduction of $R^{\mu ss}$. Denote 
$$
q_{\tilde{\mathcal U}}: \mathcal O_{Quot(H,\tau_E)} \otimes \mathcal H \rightarrow \tilde{\mathcal U}
$$
as the pull-back of the universal quotient $[q_{\mathcal U}]$ to $\mathcal{Z} \times X$. Consider the class
$$
u_{n-1}=-\rank(E) c_1(L_1)\cdots c_{1}(L_{n-1})+\chi(c_1(L_1)\cdots c_{1}(L_{n-1}).c(E))[\mathcal O_x]
$$ 
where $x\in X$ is a fixed point. Now consider the line bundle 
$$
\mathcal{L}_{n-1}:=\lambda_{\tilde{\mathcal U}}(u_{n-1})
$$
of which the higher power is a semi-ample line bundle over $\mathcal Z$. Then one can form a formal GIT quotient as 
$$
M^{\mu ss}:=\text{Proj}(\oplus_{k\geq 0} H^0(\mathcal Z, \mathcal
L_{n-1}^{\nu N})^{\SL})
$$
for some $N$. The conclusion is that this is a projective scheme with certain universal properties and the natural surjective map
$
\pi: \mathcal Z \rightarrow M^{\mu ss}
$
collapses the $\SL$ orbits and $\pi(q)=\pi(q')$ only if the sheaves $\E$ and
$\E'$ associated to $q$ and $q'$ share the same graded sheaf
$\Gr^{HNS}(\E)\cong \Gr^{HNS}(\E')$ and $\mathcal C(\E)=\mathcal{C}(\E')$. When $\dim X=2$, the converse holds.

\subsubsection{Complex structure on $\overline{M_{HYM}^*}$ induced from a continuity map $\overline{\Phi}$}
Given a stable unitary bundle over $(E, H,  \db_A)$ over $(X, \omega)$, the
most general version of the Donaldson-Uhlenbeck-Yau theorem states that
there exists a complex gauge transformation $g$ so that the unitary connection
given by $(H, g(\db_A))$ is a HYM connection that is unique up to  unitary
gauge transformations. Now this can be generalized to the case of stable
reflexive sheaf using the notion of \emph{admissible} HYM connections
(i.e.\ finite energy on the smooth locus). Suppose $[q]\in Quot$
represents a semistable torsion free sheaf $\E$. We can take the graded
sheaf $\Gr^{HNS}(\E)$ associated to a Harder-Narasimhan-Seshadri filtration of $\E$. From this we can extract canonical algebraic data as 
$$
((\Gr^{HNS}(\E))^{**},\mathcal{C}(\mathcal{\E}))
$$
from which the first factor gives a unique admissible HYM connection $A(\E)$. Here 
$$\mathcal{C}(\mathcal{\E})=\sum m_{k}^{alg}\Sigma_k$$ 
where $\Sigma_k$ is a pure codimension two subvariety of $X$ and 
$$
m_k^{alg}=h^0(\Delta,((\Gr^{HNS}(\E))^{**}/\Gr^{HNS}(\E))|_{\Delta}).
$$ 
Here $\Delta$ is a generic holomorphic transverse slice of $\Sigma_k$.

\begin{defi}
We define $\overline{M^s}$ to be the closure of $(M^{s})^{wn}$ in $M^{\mu ss}$ where $(M^s)^{wn}$ denotes the weak normalization of $M^s$. 
\end{defi}
Then we have 
\begin{thm}
There exists a continuous map 
$$\overline{\Phi}: \overline{M^s} \rightarrow \overline{M_{HYM}^*}$$ 
which restricts to the natural map
$$\Phi: (M^s)^{wn} \rightarrow (M_{HYM}^*)^{wn}.$$ More precisely, suppose $[q:\mathcal H \rightarrow \E]$ represents a point in $\overline{M^{s}}$, then $\overline{\Phi}([\E])=(A(\E),\mathcal{C}(\E))$.
\end{thm}
We very briefly explain how the proof is done and refer the reader to
\cite{GSTW:18} for more details. We fix a sequence of smooth HYM
connections $\{A_i\}$ on $E$ which sub-converges to $(A_\infty, \mathcal
C^{an} )$. By the boundedness, we can put $\E_i=(E,\db_{A_i})$ in a fixed
Quot scheme and thus obtain an algebraic limit which can behave badly in
general. More precisely, by fixing $k$ large and choosing an $L^2$
orthonormal basis for $H^0(X, \E_i(k))$, we get a sequence of elements
$[q_i]$ in the corresponding Quot scheme. Then we can take an algebraic
limit $[q_\infty]$ of $[q_i]$ in the Quot scheme. As in
\cite[Sec.\ 4]{GSTW:18},  it can be concluded that $q_\infty$ induces a sheaf inclusion $\F_\infty^{alg}\rightarrow \E_\infty$ which is an isomorphism outside some codimension two subvariety. In particular, $\E_{\infty}=(\F_\infty^{alg})^{**}$. Using the argument in 
\cite[Sec.\ 4.3]{GSTW:18}, the singular Bott-Chern formula applied to the filtration of $\mathcal{H}$ induced by $[q_\infty]$ gives 
$\mathcal{C}(\mathcal F_\infty^{alg})=\mathcal C
$.
In particular, as in \cite{GSTW:18},
this gives that the map $\overline{\Phi}$ is continuous. Given this, since all the essential algebraic geometric results \cite{GrebToma:17} used in \cite{GSTW:18} are done for multipolarizations, it is straightforward to adapt the corresponding statements in \cite{GSTW:18}
to the case of multipolarizations to obtain the following
\begin{thm}
There exists a complex structure on $\overline{M^*_{HYM}}$ which makes $\overline{M^*_{HYM}}$ an algebraic space so that the natural map $\overline{\Phi}: \overline{M^{s}}\rightarrow \overline{M^*_{HYM}}$ is an algebraic morphism. 
\end{thm}

\bibliography{../papers}

\end{document}